\documentclass[12pt]{amsart}
\usepackage{amssymb,amsfonts,amsmath}

\newcommand{\g}{\mathfrak{g}}

\newcommand{\Z}{\mathbb{Z}}
\newcommand{\A}{\mathcal{A}}
\newcommand{\gr}{\operatorname{gr}}
\newcommand{\Dcal}{\mathcal{D}}
\newcommand{\Str}{\mathcal{O}}
\newcommand{\h}{\mathfrak{h}}

\newcommand\M{\mathcal{M}}
\newcommand\Loc{\operatorname{Loc}}
\newcommand\Ext{\operatorname{Ext}}

\newcommand\GL{\operatorname{GL}}

\newcommand\red{/\!\!/\!\!/}
\newcommand\I{\mathcal{I}}

\newcommand\J{\mathcal{J}}

\newcommand\End{\operatorname{End}}

\newcommand\Hom{\operatorname{Hom}}

\newcommand\quo{/\!/}

\newcommand\C{\mathbb{C}}

\newcommand\param{\mathfrak{p}}

\newcommand\HC{\operatorname{HC}}

\newcommand{\WC}{\mathfrak{WC}}

\newcommand{\R}{\mathbb{R}}

\newcommand{\VA}{\operatorname{V}}
\newcommand{\B}{\mathcal{B}}

\newcommand{\Cat}{\mathcal{C}}
\newcommand{\OCat}{\mathcal{O}}

\newcommand{\Leaf}{\mathcal{L}}

\newcommand{\Ca}{\mathsf{C}}
\newcommand{\CW}{\mathfrak{CW}}
\newcommand{\Res}{\operatorname{Res}}
\newcommand{\Supp}{\operatorname{Supp}}
\newcommand{\bDelta}{\overline{\Delta}}
\newcommand{\Coh}{\operatorname{Coh}}
\newtheorem{Thm}{Theorem}[section]
\newtheorem{Prop}[Thm]{Proposition}
\newtheorem{Cor}[Thm]{Corollary}
\newtheorem{Lem}[Thm]{Lemma}
\theoremstyle{definition}

\newtheorem{Rem}[Thm]{Remark}

\numberwithin{equation}{section}
\author{Ivan Losev}
\address{Department
of Mathematics, Northeastern University, Boston MA 02115 USA}
\email{i.loseu@neu.edu}
\thanks{MSC 2010: Primary 16G99; Secondary 16G20,53D20,53D55}
\title{Representation theory of quantized Gieseker varieties, I}
\begin{document}
\begin{abstract}
We study the representation theory of quantizations of Gieseker moduli spaces. We describe the categories
of finite dimensional representations for all parameters and categories $\mathcal{O}$ for special values
of parameters. We find the values of parameters, where the quantizations have finite homological dimension,
and establish abelian localization theorem. We describe the two-sided ideals. Finally, we determine
annihilators of the irreducible objects in categories $\mathcal{O}$ for some special choices of one-parameter
subgroups.
\end{abstract}
\maketitle
\tableofcontents
\section{Introduction}
Our goal is to study the representation theory of quantizations of the Gieseker moduli spaces.

\subsection{Gieseker moduli space and its quantizations}
Let us explain  constructions of the Gieseker moduli spaces and of their quantizations via Hamiltonian reduction.

Pick two vector spaces $V,W$ of dimensions $n,r$, respectively. Consider the space
$R:=\mathfrak{gl}(V)\oplus \Hom(V,W)$ and a natural action of $G:=\GL(V)$ on it.
Then we can form the cotangent bundle $T^*R$, this is a symplectic vector space.
Identifying $\mathfrak{gl}(V)^*$ with $\mathfrak{gl}(V)$ and $\Hom(V,W)^*$
with $\Hom(W,V)$ by means of the trace form, we identify $T^*R$
with $\mathfrak{gl}(V)^{\oplus 2}\oplus \Hom(V,W)\oplus \Hom(W,V)$.
The action of $G$ on $T^*R$ is symplectic so we get the  moment map
$\mu:T^*R\rightarrow \g$. It can be described in two equivalent ways. First,
we have $\mu(A,B,i,j)=[A,B]-ji$. Second, the dual map $\mu^*:\g\rightarrow \C[T^*R]$
sends $\xi\in \g$ to the vector field $\xi_R$ (the infinitesimal action of $\xi$) that
can be viewed as a polynomial function on $T^*R$.

Now pick a non-trivial character $\theta$ of $G$ and consider the open subset of $\theta$-stable
points $(T^*R)^{\theta-ss}\subset T^*R$. For example, for $\theta=\det^k$ with $k>0$, the subset of semistable points
consists of all quadruples $(A,B,i,j)$ such that $\ker i$ does not contain nonzero $A$- and $B$-stable
subspaces. Then we can form the GIT Hamiltonian reduction
$\M^\theta(n,r):=\mu^{-1}(0)^{\theta-ss}/G$, this is  the Gieseker moduli space (for all choices of
$\theta$). This a smooth symplectic quasi-projective
variety of dimension $2rn$ that is a resolution of singularities of the categorical Hamiltonian reduction
$\M(n,r):=\mu^{-1}(0)\quo G$. We note that the dilation action of $\C^\times$ on $T^*R$
descends to both $\M^\theta(n,r),\M(n,r)$ (the corresponding action will be
called {\it contracting} below). The resulting grading on $\C[\M(n,r)]$
is positive meaning that $\C[\M(n,r)]=\bigoplus_{i\geqslant 0} \C[\M(n,r)]_i$
(where $\C[\M(n,r)]_i$ is the $i$th graded component) and $\C[\M(n,r)]_0=\C$.

Now let us explain how to construct  quantizations of $\M(n,r)$ meaning  filtered
associative unital algebras $\A$ with $\gr \A\xrightarrow{\sim} \C[\M(n,r)]$
(an isomorphism of graded Poisson algebras). Take $\lambda\in \C$ and
set $$\A_\lambda(n,r):=\left( D(R)/[D(R)\{x-\lambda \operatorname{tr}x, x\in \g\}]\right)^G.$$
This is a filtered algebra (the filtration is induced from the Bernstein filtration on $D(R)$,
where $\deg R=\deg R^*=1$)
and there is a natural epimorphism $\C[M(n,r)]\twoheadrightarrow\gr\A_\lambda(n,r)$
that is an isomorphism because $\mu$ is flat.

We can also consider the quantization $\A_\lambda^\theta(n,r)$ of $\M^\theta(n,r)$.
This is a sheaf (in conical topology) of filtered algebras of $\M^\theta(n,r)$ also
obtained by quantum Hamiltonian reduction, we will recall how below. Its (derived) global sections coincide
with $\A_\lambda(n,r)$. So we have the global section functor $\Gamma_\lambda^\theta:
\A_\lambda^\theta(n,r)\operatorname{-mod}\rightarrow \A_\lambda(n,r)\operatorname{-mod}$,
where we write   $\A_\lambda^\theta(n,r)\operatorname{-mod}$ for the category of
coherent $\A_\lambda^\theta(n,r)$-modules and $\A_\lambda(n,r)\operatorname{-mod}$
for the category of finitely generated $\A_\lambda(n,r)$-modules.

We note that the variety $\M^\theta(n,r)$ comes with another torus action.
Namely, let $T_0$ denote a maximal torus in $\GL(W)$. The torus $T:=T_0\times \C^\times$
acts on $R$ by $(t,z).(A,i):=(zA,ti), t\in T_0,z\in \C^\times, A\in \mathfrak{gl}(V),
i\in \Hom(V,W)$. The action naturally lifts to a Hamiltonian action  on $T^*R$
commuting with $\GL(V)$ and hence descends to $\M^\theta(n,r)$. Note that the $T$-action
on $\M^\theta(n,r)$ commutes with the contracting action of $\C^\times$.

Now let $\bar{\A}_\lambda(n,r)$ denote the algebra obtained similarly to $\A_\lambda(n,r)$
but for the $G$-module $\bar{R}=\mathfrak{sl}(V)\oplus \Hom(V,W)$. Note that $\A_\lambda(n,r)=
D(\C)\otimes \bar{\A}_\lambda(n,r)$ so most questions about the representation theory of
$\A_\lambda(n,r)$ reduce to similar questions about $\bar{\A}_\lambda(n,r)$.

\subsection{Classical cases}\label{SS_known_results}
There is one case that was studied very extensively in the past fifteen years or so: $r=1$. Here the variety $\M^\theta(n,1)$
is the Hilbert scheme $\operatorname{Hilb}^n(\C^2)$ of $n$ points on $\C^2$
and $\M(r,n)=\C^{2n}/\mathfrak{S}_n$. The quantization $\bar{\A}_\lambda(n,r)$ is the spherical subalgebra in the
Rational Cherednik algebra $H_\lambda(n)$ for the pair $(\h,\mathfrak{S}_n)$, where $\h$ is the reflection representation
of $\mathfrak{S}_n$, see \cite{GG} for details.
The representation theory of $\bar{\A}_\lambda(n,1)$ was studied, for example, in \cite{BEG,GS1,GS2,rouqqsch,KR,BE,sraco,Wilcox}.
In particular, it is known
\begin{enumerate}
\item when (=for which $\lambda$) this algebra has finite homological dimension, \cite{BE},
\item how  to classify its  finite dimensional irreducible representations, \cite{BEG},
\item how to compute characters of irreducible
modules in the so called category $\mathcal{O}$, \cite{rouqqsch},
\item how to determine the supports of these modules, \cite{Wilcox},
\item how to describe the two-sided ideals of $\bar{\A}_\lambda(n,1)$, \cite{sraco},
\item when an analog of the Beilinson-Bernstein  localization theorem holds, \cite{GS1,KR}.
\end{enumerate}
Let us  point out that there is an even more classical
special case of the algebras $\bar{\A}_\lambda(n,r)$: when $n=1$. In this case $\M^\theta(1,r)=\C^2\times T^*\mathbb{P}^{r-1}$ and $\bar{\A}_\lambda(1,r)=D^\lambda(\mathbb{P}^{r-1})$ (the algebra of $\lambda$-twisted differential
operators). (1)-(6) in this case are known and easy.

We will address analogs of (1),(2),(5),(6)  for $\bar{\A}_\lambda(n,r)$ with general $n,r$. We  prove some
results towards (3),(4) as well.

%Before we state  our main results, let us point out that there is yet another case when
%the algebra $\bar{\A}_\lambda(n,r)$ is classical, namely when $n=1$. In this case, %$\bar{\A}_\lambda(1,r)=D^\lambda(\mathbb{P}^{r-1})$,
%the algebra of $\lambda$-twisted differential operators on $\mathbb{P}^{r-1}$.

\subsection{Main results}
First, let us give answers to (1) and (6).

\begin{Thm}\label{Thm:loc}
The following is true.
\begin{enumerate}
\item The algebra $\bar{\A}_\lambda(n,r)$ has finite global dimension  if and only if $\lambda$ is not of the form $\frac{s}{m}$,
where $1\leqslant m\leqslant n$ and $-rm<s<0$.
\item For $\theta=\det$,  the abelian localization holds for $\lambda$ (i.e., $\Gamma_\lambda^\theta$
is an equivalence)  if and only if $\lambda$
is not of the form $\frac{s}{m}$, where $1\leqslant m\leqslant n$ and $s<0$.
\end{enumerate}
\end{Thm}

In fact, part (2) is a straightforward consequence of (1) and results of McGerty and Nevins, \cite{MN_ab}.

Let us proceed to classification of finite dimensional representations.

\begin{Thm}\label{Thm:fin dim}
%The following holds.
%\begin{enumerate}
%\item The sheaf $\bar{\A}_\lambda^\theta(n,r)$ has a representation supported on $\bar{\rho}^{-1}(0)$ if and only if $\lambda=\frac{s}{n}$
%with $s$ and $n$ coprime. If that is the case, then the category %$\bar{\A}_\lambda^\theta(n,r)\operatorname{-mod}_{\rho^{-1}(0)}$
%is equivalent to $\operatorname{Vect}$.
%\item
The algebra $\bar{\A}_\lambda(n,r)$ has a finite dimensional representation  if and only if $\lambda=\frac{s}{n}$
with $s$ and $n$ coprime and the homological dimension of $\bar{\A}_\lambda(n,r)$ is finite. If that is the case, then the category $\bar{\A}_\lambda(n,r)\operatorname{-mod}_{fin}$ of finite dimensional representations is equivalent to $\operatorname{Vect}$, the category of vector spaces.
%\end{enumerate}
\end{Thm}
%In fact, (2) is an easy consequence of (1) and Theorem \ref{Thm:loc}.

Now let us proceed to the description of two-sided ideals.

\begin{Thm}\label{Thm:ideals}
The following true:
\begin{enumerate}
\item If the algebra $\bar{\A}_\lambda(n,r)$ has infinite homological dimension, then it is simple.
\item Assume that $\bar{\A}_\lambda(n,r)$ has finite homological dimension and let $m$ stand for the denominator
of $\lambda$ (equal to $+\infty$ if $\lambda$ is not rational). Then there are $\lfloor n/m\rfloor$
proper two-sided ideals in $\bar{\A}_\lambda(n,r)$, all of them are prime, and they form a chain
$\bar{\A}_\lambda(n,r)\supsetneq \J_{\lfloor n/m\rfloor}\supsetneq\ldots\supsetneq \J_2\supsetneq \J_1\supsetneq \{0\}$.
\end{enumerate}
\end{Thm}

Finally, let us explain some partial results on a category $\mathcal{O}$ for $\bar{\A}^\theta_\lambda(n,r)$, we will
recall necessary definitions below in Section \ref{SS_Cat_O}. We use the notation $\mathcal{O}_\nu(\A^{\theta}_{\lambda}(n,r))$ for this category. Here $\nu$ is a co-character of
$T$ that is supposed to be generic meaning that $\M^\theta(n,r)^{\nu(\C^\times)}$
is finite (in which case, it is in a natural bijection with the set of  $r$-multipartitions
of $n$). The co-character $\nu$ fails to be generic precisely when it lies on a finite union
of suitable hyperplanes in $\Hom(\C^\times,T)$, we will describe them explicitly below.

For now,  we need to know   the following about the category $\mathcal{O}_\nu(\A^{\theta}_{\lambda}(n,r))$:
\begin{itemize}
\item The category $\mathcal{O}_\nu(\A^{\theta}_{\lambda}(n,r))$ is a highest weight category so it makes sense to speak about standard objects $\Delta_\nu^\theta(p)$.
\item The labeling set for standard objects is naturally identified with $\M^\theta(n,r)^{T}$,
 i.e., the set of $r$-multipartitions of $n$.
\end{itemize}

\begin{Thm}\label{Thm:cat_O_easy}
If the denominator of $\lambda$ is bigger than $n$, then the category $\mathcal{O}_\nu(\A^\theta_\lambda(n,r))$ is semisimple. If the denominator of $\lambda$ equals $n$, the category $\mathcal{O}_\nu(\A^\theta_\lambda(n,r))$ has only one nontrivial block. That block is equivalent to the nontrivial block of $\mathcal{O}_\nu(\A^\theta_{1/nr}(nr,1))$ (a.k.a. the category of $B$-equivariant perverse sheaves on $\mathbb{P}^{nr-1}$).
\end{Thm}
In some cases, we can say which simple objects belong to the nontrivial block, we will do this below.

%Let us state a result about two-sided ideals.
%
%\begin{Thm}\label{Thm:ideals}
%Let $m$ denote the denominator of $\lambda$ and suppose that $m\leqslant n$ and $\lambda\not\in (-r,0)$.
%Then there are $k:=\lfloor n/m\rfloor$ two-sided ideals, all of them are prime and they form
%a chain: $\J_0:=0\subsetneq \J_1\subsetneq \J_2\subsetneq\ldots \subsetneq \J_k\subsetneq \A_\lambda(n,r)$.
%Moreover, the GK dimension of $\A_\lambda(n,r)/\J_i$ is $2(rm-1)(i-1)$.
%\end{Thm}

We can also determine the annihilators of the simple modules in the category $\mathcal{O}$. In this paper
we treat a special co-character $\nu:\C^\times\rightarrow T$ to be called {\it dominant} below.
Namely, we pick integers $d_1,\ldots,d_r$ with $d_1\gg d_2\gg\ldots\gg d_r$.  Consider
$\nu:\C^\times\rightarrow T$ given by $t\mapsto \left((t^{d_1},\ldots,t^{d_r}),t\right)$,
it is easy to see that it is generic (in fact, we can just take $d_i-d_{i+1}>n$).
We also do not consider the case when $\lambda$ is integral, in this case the category
was described in \cite{Webster_O}.

Now take a multipartition $\tau=(\tau^{(1)},\ldots,\tau^{(r)})$. Let $m$ denote the denominator of
$\lambda$, we assume that $m>1$. We assume that $\lambda>0$ so that the functor $\Gamma_\lambda^\theta$
is an equivalence. Divide $\tau^{(1)}$ by $m$ with remainder: $\tau^{(1)}=m\tau'+\tau''$, where $\tau',\tau''$
are partitions with maximal possible $|\tau'|$, and the sum and the multiplication by $m$ are defined component-wise. For
example, if  $m=3$ and $\tau=(8,6,1)$, then $\tau'=(1^2)$ and $\tau''=(5,3,1)$.

\begin{Thm}\label{Thm:support}
Assume that $m>1$. Under the assumptions above, the annihilator of $\Gamma(L^\theta_\nu(\tau))$ (where
$L^\theta_\nu(\tau)$ denotes the irreducible module labeled by $\tau$) coincides
with the ideal $\J_{|\tau'|}$ from Theorem \ref{Thm:ideals}.
\end{Thm}

We also give some results towards a computation of
the annihilators of $L^\theta_\nu(\tau)$ for non-dominant $\nu$.
Namely, for each two adjacent chambers $C,C'$ in $\Hom(\C^\times,T)$, there is a
{\it cross-walling} bijection $\mathfrak{cw}^\lambda_{C\rightarrow C'}:\M^\theta(n,r)^T
\rightarrow \M^\theta(n,r)^T$ that preserves the annihilator of the simple modules.
We will define  this bijection  below.

{\bf Acknowledgments}. I would like to thank Roman Bezrukavnikov, Dmitry Korb, Davesh Maulik, Andrei Okounkov and Nick
Proudfoot for stimulating discussions. I also would like to thank Boris Tsvelikhovsky for comments on an 
earlier version of this text. My work was supported by the NSF  under Grant DMS-1501558.

%We will also study the so called {\it wall-crossing} functors for the algebras $\bar{\A}_\lambda(n,r)$. This is
%a crucial ingredient to the application to the representation theory of the general algebras $\A_\lambda(v,w)$
%to be described in the next subsection.
\section{Preliminaries}
\subsection{Symplectic resolutions and their quantizations}\label{SS_sympl_sing_quant}
Let $X_0$ be a normal Poisson affine variety equipped with an action of $\C^\times$
such that the grading on $\C[X_0]$ is positive (meaning that the graded component $\C[X_0]_i$
is zero when $i<0$, and $\C[X_0]_0=\C$) and there is a positive integer $d$
such that $\{\C[X_0]_i,\C[X_0]_j\}\subset \C[X_0]_{i+j-d}$. By a symplectic resolution
of singularities of $X_0$ we mean a pair $(X,\rho)$ of
\begin{itemize}
\item a smooth symplectic algebraic variety $X$ (with form $\omega$)
\item a morphism $\rho:X\rightarrow X_0$ of Poisson varieties that is a projective resolution of singularities.
\end{itemize}
Below we assume that $(X,\rho)$ is a symplectic resolution of singularities. Besides, we will
assume that $(X,\rho)$ is {\it conical} meaning that the $\C^\times$-action lifts to $X$
in such a way that $\rho$ is equivariant. This $\C^\times$-action will be called {\it contracting}
later on.

Note that  $\rho^*:\C[X_0]\rightarrow \C[X]$ is an isomorphism because $X_0$ is normal.
By the Grauert-Riemenschneider theorem, we have
$H^i(X,\mathcal{O}_X)=0$ for $i>0$.
%By results of Kaledin, \cite[Theorem 2.3]{Kaledin}
%, $X_0$ has finitely
%many symplectic leaves.

If $X,X'$ are two conical symplectic resolutions of $X_0$, then the Pickard groups of $X,X'$
are naturally identified, see, e.g.,
%there are
%open subvarieties $\breve{X}\subset X, \breve{X}'\subset X'$ with $\operatorname{codim}_X X\setminus \breve{X},
%\operatorname{codim}_X X\setminus \breve{X}\geqslant 2$ and $\breve{X}\xrightarrow{\sim} \breve{X}'$,
%see, e.g.,
\cite[Proposition 2.19]{BPW}. %This allows to identify the Picard groups $\operatorname{Pic}(X)=\operatorname{Pic}(X')$.
Moreover, the Chern class map defines an isomorphism $\C\otimes_{\Z}\operatorname{Pic}(X)\xrightarrow{\sim} H^2(X,\C)$,
%The isomorphisms $\C\otimes_{\Z}\operatorname{Pic}(X)\xrightarrow{\sim} H^2(X,\C)$,
%$\C\otimes_{\Z}\operatorname{Pic}(X')\xrightarrow{\sim} H^2(X',\C)$ intertwine the identifications
%$\operatorname{Pic}(X)\cong \operatorname{Pic}(X'), H^2(X,\C)\cong H^2(X',\C)$.
see, e.g., \cite[Section 2.3]{BPW}. Let us write $\tilde{\param}$ for $H^2(X,\C)$ and let $\tilde{\param}_{\Z}$
denote the image of $\operatorname{Pic}(X)$ in $H^2(X,\C)$.

Set $\tilde{\param}_{\R}:=\R\otimes_{\Z}\tilde{\param}_{\Z}$. According to Namikawa, \cite{Namikawa},
there is a linear hyperplane arrangement in $\tilde{\param}_{\R}$ together with an action of
a crystallographic reflection group $W$ subject to the following conditions:
\begin{itemize}
\item The walls for $W$ are in the arrangement.
\item The conical symplectic resolutions of $X$ are classified by $W$-conjugacy classes of chambers.
\end{itemize}
For $\theta$ inside a chamber, we will write $X^\theta$ for the corresponding resolution.
%There is a finite group $W$ acting on
%$\tilde{\param}_{\R}$ as a reflection group, such that the movable cone $C$ of $X$ (that does not
%depend on the choice of a resolution) is a fundamental chamber for $W$. We can find hyperplanes
%$H_1,\ldots,H_k$ that are either walls of $C$ or pass through the interior of $C$ partitioning $C$
%into the union of cones $C_1,\ldots,C_m$ such that the possible conical symplectic resolutions of
%$X_0$ are in one-to-one correspondence with $C_1,\ldots,C_m$ in such a way that the cone corresponding
%to a resolution $X'$ is its  ample cone.  Set
%\begin{equation}\label{eq:singular} \mathcal{H}_{\C}:=\bigcup_{1\leqslant i\leqslant k,w\in %W}w(\C\otimes_{\R}H_i).\end{equation}
%Then $X_\lambda$ is affine if and only if $\lambda\not\in \mathcal{H}_{\C}$.
%The results in this paragraph are due to Namikawa, \cite{Namikawa13}, see
%also \cite[Section 2.3]{BPW} for an exposition.

%\begin{defi}\label{defi:chamb_termin}
%We say that an element $\theta\in \tilde{\param}_{\Q}$ is {\it generic} if it does not lie in $\mathcal{H}_\C$.
%The cones $wC_i$ will be called {\it chambers}.
%\end{defi}
%
%For $\theta\not\in \mathcal{H}_{\C}$, we will write $X^\theta$ for the resolution corresponding to the ample cone
%containing $W\theta\cap C$. Further, if $w\theta\in C$, we will choose a different identification
%of $H^2(X^\theta,\C)$ with $\tilde{\param}$, one twisted by $w$ (so that the ample cone actually
%contains $\theta$).

We will study quantizations of $X_0$ and $X$. By a quantization of $X_0$, we mean a filtered
algebra $\A$ together with an isomorphism $\gr\A\cong \C[X_0]$ of graded Poisson algebras.

By a quantization of $X=X^\theta$, we mean a sheaf $\A^\theta$ of filtered algebras in the conical
topology on $X$ (in this topology, ``open'' means Zariski open and $\C^\times$-stable) that is complete
and separated with respect to the filtration together with an isomorphism $\gr\A^\theta\cong \Str_{X^\theta}$
(of sheaves of graded Poisson algebras). A result of Bezrukavnikov and Kaledin, \cite{BK}, (with variations given in \cite[Section 2.3]{quant}) shows that quantizations $\A^\theta$ are parameterized (up to an isomorphism) by the points in $\tilde{\param}$. We write $\A^\theta_\lambda$ for the quantization corresponding to $\lambda\in \tilde{\param}$.
Note  that $\A^\theta_{-\lambda}$ is isomorphic to $(\A_\lambda^\theta)^{opp}$, this follows
from \cite[Section 2.3]{quant}.

We set $ \A_\lambda=\Gamma(\A_\lambda^\theta)$. It follows
from \cite[Section 3.3]{BPW} that the algebras $\A_\lambda$ are independent from the choice of $\theta$.
From $H^i(X^\theta,\Str_{X^\theta})=0$, we deduce that the higher cohomology
 of $\A^\theta_\lambda$ vanishes and  that  $\A_\lambda$ is a quantization of $\C[X]=\C[X_0]$. By the previous paragraph, $\A_{-\lambda}\cong \A_\lambda^{opp}$.
%Also we have $\A_\lambda\cong \A_{w\lambda}$ for all $\lambda\in \tilde{\param},w\in W$,
%see \cite[Section 3.3]{BPW}.

Let us compare the categories $\A_\lambda\operatorname{-mod}$ of finitely generated $\A_\lambda$-modules
and $\operatorname{Coh}(\A_\lambda^\theta)$ of coherent sheaves of $\A_\lambda^\theta$-modules.
We have functors $\Gamma:
\operatorname{Coh}(\A_\lambda^\theta)\rightarrow \A_\lambda\operatorname{-mod}$
of taking global sections and its left adjoint, the localization functor $\operatorname{Loc}$.
When we need to indicate the dependence on $(\lambda,\theta)$, we write $\Gamma^\theta_\lambda,
\Loc_\lambda^\theta$. We say that that $(\lambda,\theta)$ satisfy abelian (resp., derived)
localization if the functors $\Gamma_\lambda^\theta, \Loc_\lambda^\theta$ are mutually inverse
(resp., if the corresponding derived functors $R\Gamma_\lambda^\theta, L\Loc_\lambda^\theta$
are mutually inverse).

The following result was obtained in \cite[Section 5.3]{BPW}.

\begin{Lem}\label{Lem:lem_ab_equiv}
Let $\chi\in \tilde{\param}_{\Z}$ be ample for  $X^\theta$  and let $\lambda\in \tilde{\param}$.
Then there is $n_0\in \Z$ such that $(\lambda'+n\chi,\theta)$
satisfies abelian localization for all $n\geqslant n_0$.
\end{Lem}

\subsection{Category $\mathcal{O}$}\label{SS_Cat_O}
Suppose that we have a conical symplectic resolution $X=X^\theta$ that
comes equipped with a Hamiltonian action of a torus $T$
that commutes with the contracting $\C^\times$-action. Let $\lambda\in \tilde{\param}$. The action of
$T$ on $\mathcal{O}_X$ lifts to a Hamiltonian action of $T$ on $\A_\lambda^\theta$.
So we get a Hamiltonian action on $\A_\lambda$.
By $\Phi$ we denote the quantum comoment map $\mathfrak{t}\rightarrow \A_\lambda$.

Let $\nu:\C^\times\rightarrow T$ be a one-parameter subgroup. The subgroup $\nu$ induces
a grading $\A_\lambda=\bigoplus_{i\in \Z}\A_\lambda^{i,\nu}$. We set $\A_\lambda^{\geqslant 0,\nu}=\bigoplus_{i\geqslant 0}\A_\lambda^{i,\nu}$ and define $\A_\lambda^{>0,\nu}$ similarly. Further, set $\Ca_{\nu}(\A_\lambda):=\A^{0,\nu}_\lambda/\bigoplus_{i>0}\A^{-i,\nu}_\lambda \A^{i,\nu}_\lambda$.
Note that $\A_\lambda/\A_\lambda\A_{\lambda}^{>0,\nu}$ is an $\A_\lambda$-$\Ca_{\nu}(\A_\lambda)$-bimodule,
while $\A_\lambda/\A_{\lambda}^{<0,\nu}\A_\lambda$ is a $\Ca_{\nu}(\A_\lambda)$-$\A_\lambda$-bimodule.

Define the category $\OCat_{\nu}(\A_\lambda)$ as the full subcategory of $\A_\lambda\operatorname{-mod}$
consisting of all modules, where the action of $\A_{\lambda}^{>0,\nu}$ is locally nilpotent.
We get two functors $\Delta_{\nu},\nabla_{\nu}:\Ca_{\nu}(\A_\lambda)\operatorname{-mod}
\rightarrow \OCat_{\nu}(\A_\lambda)$ given by $$\Delta_{\nu}(N):=(\A_\lambda/\A_{\lambda}\A_{\lambda}^{>0,\nu})\otimes_{\Ca_{\nu}(\A_\lambda)}N,
\nabla_{\nu}(N):=\Hom_{\Ca_{\nu}(\A_\lambda)}(\A_\lambda/\A_{\lambda}^{<0,\nu}\A_\lambda, N).$$
Here we consider the restricted Hom (with respect to the natural grading on $\A_\lambda/\A_{\lambda}^{<0,\nu}\A_\lambda$).

Now suppose that $|X^T|<\infty$. We say that a one-parameter
group $\nu:\C^\times\rightarrow T$ is {\it generic} if $X^{\nu(\C^\times)}=X^T$.
Equivalently, $\nu$ is generic if and only if it does not lie in $\ker\kappa$ for any character $\kappa$ of the
$T$-action on $\bigoplus_{p\in X^T}T_pX$. The hyperplanes $\ker\kappa$ split the
lattice $\Hom(\C^\times,T)$ into the union of polyhedral regions to be called
{\it chambers} (of one-parameter subgroups).

Suppose that $\nu$ is generic. Further, pick a generic element
$\theta\in \tilde{\param}_{\Z}$ with $X=X^\theta$ and $\lambda_0\in \tilde{\param}$.
Let $\lambda:=\lambda_0+n\theta$ for $n\gg 0$.  The following results were obtained in
\cite{BLPW} and \cite{CWR},
see \cite[Proposition 4.1]{WC_res} for precise references.

\begin{Prop}\label{Prop:cat_O}
The following is true:
\begin{enumerate}
\item The category $\OCat_\nu(\A_\lambda)$ only depends on the chamber of $\nu$.
\item The natural functor $D^b(\OCat_\nu(\A_\lambda))\rightarrow D^b(\A_\lambda\operatorname{-mod})$
is a full embedding.
\item $\Ca_{\nu}(\A_\lambda)=\C[X^T]$.
\item More generally, we have $\Ca_{\nu_0}(\A_\lambda)=\bigoplus_{Z} \A^Z_{\iota^*_Z(\lambda)-\rho_Z}$,
where the summation is taken over the irreducible components $Z$ of $X^{\nu_0(\C^\times)}$,
$\iota_Z$ is the embedding $Z\hookrightarrow X$, $\iota_Z^*:H^2(X,\C)\rightarrow H^2(Z,\C)$
is the corresponding pull-back map, $\rho_Z$ equals to a half of the 1st Chern class
of the contracting bundle of $Z$,
and $\A_{\iota^*_Z(\lambda)-\rho_Z}^Z$ stands for the global sections of the filtered quantization of $Z$ with period
$\iota^*_Z(\lambda)-\rho_Z$.
\item The category $\OCat_\nu(\A_\lambda)$ is highest weight, where the standard objects are
$\Delta_\nu(p)$, the costandard objects are $\nabla_{\nu}(p)$, where $p\in X^T$.
For an order, which is a part of the definition of a highest weight structure, we take
the contraction order on $X^T$ given by $\nu$.
\item Suppose $\nu_0$ lies in the face of a chamber containing $\nu$. Then $\Delta_{\nu_0},
\nabla_{\nu_0}$ restrict to exact functors $\OCat_\nu(\Ca_{\nu_0}(\A_\lambda))\rightarrow
\OCat_\nu(\A_\lambda)$. Moreover, $\Delta_\nu=\Delta_{\nu_0}\circ \underline{\Delta}$
and $\nabla_\nu=\nabla_{\nu_0}\circ \underline{\nabla}$, where we write $\underline{\Delta}$
and $\underline{\nabla}$ are the standardization and costandardization functors
for $\OCat_\nu(\Ca_{\nu_0}(\A_\lambda))$.
%\item The functor $\WC^{-1}_{\lambda\leftarrow \lambda^-}:D^b(\OCat_\nu(\A_\lambda))\rightarrow
%D^b(\OCat_{\nu}(\A_{\lambda^-}))$ is a Ringel duality functor.
\end{enumerate}
\end{Prop}
Let us explain what we mean by the contracting bundle in (4). This is the subvariety in
$X$ consisting of all points $x$ such that $\lim_{t\rightarrow 0}\nu_0(t)x$ exists and
lies in $Z$. Highest weight categories mentioned in (5) will be recalled in
the next section.

Now let us mention the holonomicity property, see \cite[Section 4.4]{CWR} and \cite[Theorems 1.2,1.3]{B_ineq}.

\begin{Lem}\label{Lem:O_holon}
Every module from category $\mathcal{O}_\nu(\A_\lambda)$ is holonomic in the sense of \cite{B_ineq}.
In particular, if $M$ is a simple object in $\mathcal{O}_\nu(\A_\lambda)$, then
$\operatorname{GK-}\dim M=\frac{1}{2}\operatorname{GK-}\dim (\A_\lambda/\operatorname{Ann}M)$.
\end{Lem}

%Let us finish this section by describing the opposite of $\OCat_\nu(\A_\lambda)$.
%
%\begin{Lem}\label{Lem:O_opp}
%There is an equivalence $\OCat_\nu(\A_\lambda)^{opp}\cong \OCat_{-\nu}(\A_\lambda^{opp})$
%that preserves the annihilators and GK dimensions of simple objects.
%\end{Lem}
%\begin{proof}
%The equivalence is given by taking the restricted dual, see \cite[Section 4.2]{CWR}.
%It preserves the annihilators and hence by Lemma \ref{Lem:O_holon}, the GK dimensions.
%\end{proof}

We will also need the following important property, \cite[Lemma 6.4]{BLPW}.

\begin{Lem}\label{Lem:cost_classes}
The classes of standard and costandard objects in $K_0(\OCat_\nu(\A_\lambda^\theta))$ coincide.
\end{Lem}

\subsection{Standardly stratified structures and cross-walling functors}
Let us start by recalling standardly stratified categories following \cite{LW}.

Let $\mathcal{C}$ be a $\C$-linear abelian category equivalent to the category of finite dimensional
representations of a finite dimensional algebra. Let $\mathcal{T}$ be an indexing set of
the simple objects in $\mathcal{C}$. We write $L(\tau),P(\tau)$ for the simple
and indecomposable projective objects indexed by $\tau\in \mathcal{T}$.
The additional structure of a standardly stratified category on $\mathcal{C}$ is a
partial  {\it pre-order} $\leqslant$ on $\mathcal{T}$ that should satisfy certain axioms to be explained below.
Let us write $\Xi$ for the set of equivalence classes of $\leqslant$,
this is a poset (with partial order again denoted by $\leqslant$) that comes with a natural surjection $\varrho:\mathcal{T}\twoheadrightarrow \Xi$. The pre-order $\leqslant$ defines a filtration on $\Cat$ by Serre subcategories indexed by $\Xi$. Namely, to $\xi\in \Xi$ we assign the subcategories $\Cat_{\leqslant \xi}$ that is the Serre span
of the simples $L(\tau)$ with $\varrho(\tau)\leqslant \xi$. Define $\Cat_{<\xi}$ analogously and let
$\Cat_\xi$ denote the quotient $\Cat_{\leqslant \xi}/\Cat_\xi$. Let $\pi_\xi$ denote the quotient
functor $\Cat_{\leqslant \xi}\twoheadrightarrow \Cat_{\xi}$. Let us write $\Delta_\xi:\Cat_\xi\rightarrow \Cat_{\leqslant \xi}$ for the left adjoint functor of $\pi_\xi$. Also we write $\gr\Cat$ for $\bigoplus_{\xi}\Cat_\xi, \Delta$
for $\bigoplus_\xi \Delta_\xi:\gr\Cat\rightarrow \Cat$. We call $\Delta$ the
{\it standardization functor}. Finally, for $\tau\in \varrho^{-1}(\xi)$
we write $L_\xi(\tau)$ for $\pi_\xi(L(\tau))$, $P_\xi(\tau)$ for the projective cover of
$L_\xi(\tau)$ in $\Cat_\xi$ and $\Delta(\tau)$ for $\Delta_\xi(P_\xi(\tau))$.
The object $\Delta(\tau)$ is called {\it standard}. The object $\bDelta(\tau):=\Delta_\xi(L_\xi(\tau))$
is called {\it proper standard}. Note that there is a natural epimorphism $P(\tau)\twoheadrightarrow
\Delta(\tau)$.

The axioms to be satisfied by $(\Cat,\leqslant)$ in order to give a standardly stratified structure are as follows.
\begin{itemize}
\item[(SS1)] The functor $\Delta:\gr\Cat\rightarrow \Cat$ is exact.
\item[(SS2)] The projective $P(\tau)$ admits an epimorphism onto $\Delta(\tau)$ whose kernel
has a filtration with successive quotients $\Delta(\tau')$, where $\tau'>\tau$.
\end{itemize}

When $\Cat_\xi=\operatorname{Vect}$ for all $\xi$ (in which case $\Xi=\mathcal{T}$), we recover the
classical notion of a highest weight category.

Note that $\Cat^{opp}$ is also a standardly stratified category with the same poset $\Xi$, \cite[Section 2.2]{LW}.
So we have the exact costandardization functor $\nabla_\xi$, the right adjoint of $\pi_\xi$.

Let us describe a standardly stratified  structure on $\OCat_\nu(\A_\lambda)$ (where $\lambda$
is as in Proposition \ref{Prop:cat_O}) that comes from a one-parameter subgroup $\nu_0$
lying in a face of the chamber containing $\nu$.
Then $\nu_0$ defines the order on the set of irreducible components
of $X^{\nu_0(\C^\times)}$ (by contraction, see \cite[Section 6.1]{CWR} for details).
So we get a pre-order $\leqslant_{\nu_0}$ on the set $X^T$.
It is easy to see (and was checked in \cite[Section 6.1]{CWR}) that the order
$\leqslant_{\nu}$ refines $\leqslant_{\nu_0}$.

The following proposition is the main result of \cite[Section 6]{CWR}.

\begin{Prop}\label{Prop:stand_filt}
Let $\lambda$ be as in Proposition \ref{Prop:cat_O}.
Then the pre-order $\leqslant_{\nu_0}$ defines a standardly stratified structure on
$\OCat_\nu(\A_\lambda)$. The associated graded category is $\OCat_\nu(\Ca_{\nu_0}(\A_\lambda))$.
The standardization functor is $\Delta_{\nu_0}$ and the costandardization functor is $\nabla_{\nu_0}$.
\end{Prop}

Let us now discuss cross-walling functors that are derived equivalences
between categories $\mathcal{O}$ corresponding to different generic
one-parameter subgroups. Namely, let $\nu,\nu'$ be two generic one-parameter
subgroups and $\lambda$ be as in Proposition \ref{Prop:cat_O}. The following
proposition is established in \cite[Section 7]{CWR}.

\begin{Prop}\label{Prop:cw_functors}
The following is true:
\begin{enumerate}
\item There is an equivalence $\mathfrak{CW}_{\nu'\leftarrow \nu}:D^b(\OCat_\nu(\A_\lambda))
\rightarrow D^b(\OCat_{\nu'}(\A_\lambda))$ such that we have a bifunctorial isomorphism
$$\Hom_{D^b(\OCat_{\nu'}(\A_\lambda))}(\mathfrak{CW}_{\nu'\leftarrow \nu}(M),N)\xrightarrow{\sim}
\Hom_{D^b(\A_\lambda\operatorname{-mod})}(M,N),$$
where $M\in D^b(\OCat_\nu(\A_\lambda)),N\in D^b(\OCat_{\nu'}(\A_\lambda))$.
\item Suppose that $\nu_0$ lies in a common face of the chambers containing $\nu,\nu'$.
Then we have an isomorphism $\mathfrak{CW}_{\nu'\leftarrow \nu}\circ \Delta_{\nu_0}\xrightarrow{\sim}
\Delta_{\nu_0}\circ \underline{\mathfrak{CW}}_{\nu'\leftarrow \nu}$, where we write
$\underline{\mathfrak{CW}}_{\nu'\leftarrow \nu}$ for the cross-walling functor
$D^b(\OCat_\nu(\Ca_{\nu_0}(\A_\lambda)))\rightarrow D^b(\OCat_{\nu'}(\Ca_{\nu_0}(\A_\lambda)))$.
\item The functor $\mathfrak{CW}_{-\nu\leftarrow \nu}[-\frac{1}{2}\dim X]$ is the
\text{Ringel duality functor}, i.e., an equivalence that maps $\Delta_{\nu}(p)$ to
$\nabla_{-\nu}(p)$ for all $p\in X^T$.
\end{enumerate}
\end{Prop}

%\subsection{Harish-Chandra bimodules and restriction functors}
%Now let $\A^j, i=1,2,$ be an associative algebra with a  filtration $\A^j=\bigcup_{i\geqslant 0}\A^j_{\leqslant i}$.
%Suppose that the algebras $\gr\A^j, j=1,2$ are finitely generated and commutative and they are identified,
%let us denote this algebra by $A$.
%
%By a Harish-Chandra $\A^1$-$\A^2$-bimodule we mean a bimodule $\B$ that can be equipped with a
%bimodule filtration $\B=\bigcup_{i\in \Z}\B_{\leqslant i}$ such that
%\begin{itemize}
%%\item The filtration is compatible with the action of $\Z/d\Z$ on $\B$.
%\item If $a^j\in \A^j_{\leqslanrt i}$ are such that their top degree components coincide
%and $b\in \B_{\leqslant k}$, then $a^1b-ba^2\in \B_{\leqslant i+k-1}$.
%\item $\gr\B$ is a finitely generated $\gr\A$-module.
%\end{itemize}
%
%For a HC bimodule $\B$ we can define its associated variety $\VA(\B)$
%inside of the reduced scheme associated to $A$ in the usual way.
%Note that $\VA(\B)$ is a Poisson subvariety of $\operatorname{Spec}(B)$.
%
%Let us give an example of a Harish-Chandra bimodule. Pick $\chi\in \tilde{\param}_{\Z},\lambda\in
%\tilde{\param}$

\subsection{Wall-crossing functors and bijections}\label{SS_long_WC}
Let $\theta,\theta'$ be two generic elements of $\tilde{\param}_{\Z}$ and $\lambda\in \tilde{\param}$.
Following \cite[Section 6.3]{BPW}, we are going to produce a derived equivalence $\WC_{\theta'\leftarrow \theta}:D^b(\Coh(\A_\lambda^\theta))\xrightarrow{\sim}
D^b(\Coh(\A_\lambda^{\theta'}))$ assuming abelian localization holds for $(\lambda,\theta)$
and derived localization holds for  $(\lambda,\theta')$. Then we set  $\WC_{\theta'\leftarrow \theta}:=L\Loc_{\lambda}^{\theta'}\circ \Gamma_\lambda^\theta$. Note that this functor
is right t-exact.

We can give a different realization of $\WC_{\theta'\leftarrow \theta}$. Namely, pick $\lambda'\in
\lambda+\tilde{\param}_{\Z}$ such that abelian localization holds for $(\lambda',\theta')$. We identify
$\Coh(\A_{\lambda}^\theta)$ with $\A_\lambda\operatorname{-mod}$ by means of $\Gamma_{\lambda}^\theta$
and $\Coh(\A_{\lambda}^{\theta'})$ with $\A_{\lambda'}\operatorname{-mod}$ by means of $\Gamma_{\lambda'}^{\theta'}(
\A^{\theta'}_{\lambda,\lambda'-\lambda}\otimes_{\A_\lambda^{\theta'}}\bullet)$, where $\A^{\theta'}_{\lambda,\lambda'-\lambda}$ is the $\A_{\lambda'}^{\theta'}$-$\A_\lambda^{\theta'}$-bimodule
quantizing the line bundle corresponding to $\lambda'-\lambda$ on $X^{\theta'}$. Under these identifications,
the functor $\WC_{\theta'\leftarrow \theta}$ becomes $\WC_{\lambda'\leftarrow \lambda}:=\A_{\lambda,\lambda'-\lambda}^{(\theta)}\otimes^L_{\A_\lambda}\bullet$, see \cite[Section 6.4]{BPW}.
Here we write $\A_{\lambda,\lambda'-\lambda}^{(\theta)}$ for the global sections of $\A_{\lambda,\lambda'-\lambda}^{\theta}$.

We note that the functor $\WC_{\lambda'\leftarrow \lambda}$ restricts to an equivalence
$D^b(\OCat_\nu(\A_\lambda))\rightarrow D^b(\OCat_\nu(\A_{\lambda'}))$, see \cite[Section 8.1]{BLPW}.

We are especially interested in the situation when $\theta,\theta'$ lie in the opposite chambers.
In this case the {\it long wall-crossing} functor $\WC_{\lambda'\leftarrow \lambda}$   is perverse in
the sense of Chuang and Rouquier, see \cite[Section 3]{WC_res}. Let us recall the general definition.

Let us recall the general definition.
Let $\mathcal{T}^1,\mathcal{T}^2$ be  triangulated categories equipped with  $t$-structures
that are homologically finite. Let $\mathcal{C}^1,\mathcal{C}^2$ denote the hearts of $\mathcal{T}^1,\mathcal{T}^2$, respectively.

We  are going to define a perverse equivalence with respect to  filtrations
$\mathcal{C}^i=\mathcal{C}_0^i\supset \mathcal{C}_1^i
\supset\ldots \supset\mathcal{C}_k^i=\{0\}$ by Serre subcategories. By definition, this is a triangulated equivalence
$\mathcal{F}:\mathcal{T}^1\rightarrow \mathcal{T}^2$ subject to the following
conditions:
\begin{itemize}
\item[(P1)] For any $j$, the equivalence $\mathcal{F}$ restricts to an equivalence
$\mathcal{T}^1_{\mathcal{C}_j^1}\rightarrow \mathcal{T}^2_{\mathcal{C}_j^2}$, where
we write $\mathcal{T}^i_{\mathcal{C}_j^i}, i=1,2,$ for the category of all objects
in $\mathcal{T}^i$ with homology (computed with respect to the t-structures of interest)
in $\mathcal{C}_j^i$.
\item[(P2)] For $M\in \Cat_j^1$, we have $H_\ell(\mathcal{F}M)=0$ for $\ell<j$
and $H_\ell(\mathcal{F}M)\in \Cat^2_{j+1}$ for $\ell>j$.
\item[(P3)] The  functor $M\mapsto H_j(\mathcal{F}M)$ induces an equivalence $\Cat^1_j/\Cat^1_{j+1}\xrightarrow{\sim}
\Cat^2_j/\Cat^2_{j+1}$  of abelian categories.
\end{itemize}

\begin{Lem}\label{Lem:long_wc_perv}
The equivalence $\WC_{\lambda'\leftarrow \lambda}: D^b(\OCat_{\nu}(\A_{\lambda}))\rightarrow D^b(\OCat_\nu(\A_{\lambda'}))$
is perverse with $\Cat^1_j=\{M\in \OCat_\nu(\A_\lambda)| \operatorname{GK-}\dim M\leqslant \dim X/2-j\}$
(and $\Cat^2_j$ defined similarly).
\end{Lem}

%Let us explain how the filtration looks in the case of wall-crossing functors.
%First, let us explore the
%case when $\theta'=-\theta$. It was shown in \cite{B_ineq}, that the $\A_{\lambda}$-$\A_\lambda$ bimodule
%$\A_{\lambda}$ has finite length. In particular, for any $d$, there is a minimal ideal $\I_{\lambda,d}
%\subset \A_\lambda$ such that $\dim \VA(\A_\lambda/\I_{\lambda,d})\leqslant 2d$. Then, for
%$\mathcal{C}^1=\A_{\lambda}\operatorname{-mod}$, we have $\mathcal{C}^1_{\dim %X/2-d}=\A_\lambda/\I_{\lambda,d}\operatorname{-mod}$.
%The similar description applies to $\A_{\lambda'}\operatorname{-mod}$. For $\Cat^1=\OCat_\nu(\A_\lambda)$,
%the description simplifies: $\mathcal{C}^1_{\dim X/2-d}=\{M\in \OCat_\nu(\A_\lambda)| \operatorname{GK-}\dim M\leqslant %d\}$.

%In the general case, the filtration is also given by the annihilation by certain ideal, we refer the reader
%to \cite[Section 3.1]{WC_res}.

Now let us discuss the wall-crossing bijections. (P3) gives rise to a bijection $\operatorname{Irr}(\mathcal{C}^1)\rightarrow \operatorname{Irr}(\mathcal{C}^2)$
that is called a {\it wall-crossing bijection} and denoted by $\mathfrak{wc}_{\lambda'\leftarrow
\lambda}$ when we deal with the wall-crossing functors $\WC_{\lambda'\leftarrow \lambda}$.
%For categories $\mathcal{O}$, this bijection enjoys the following important property
%established in \cite{WC_res}.
%
%\begin{Prop}\label{Prop:wc_bij}
%Let $\psi$ be a Weil generic element of the linear span of the largest common face of $C,C'$.
%Then the self-bijections $\mathfrak{wc}_{\lambda'\leftarrow \lambda},\mathfrak{wc}_{\lambda'+\psi\leftarrow
%\lambda+\psi}$ of $X^T$ coincide.\end{Prop}

Finally, let us recall the following important property of $\WC_{\lambda'\leftarrow \lambda}$ established in \cite[Section 7.3]{CWR}.

\begin{Lem}\label{Lem:WC_Ringel}
In the case when $\theta=-\theta'$, the functor $\WC_{\lambda'\leftarrow \lambda}:
D^b(\OCat_\nu(\A_\lambda))\rightarrow D^b(\OCat_\nu(\A_{\lambda'}))$ is a
Ringel duality functor.
\end{Lem}

\section{Gieseker varieties and their quantizations}
The definition of the Gieseker varieties $\M(n,r),\M^\theta(n,r)$ (for $\theta=\det$ or $\det^{-1}$)
was recalled in the introduction. We note that $\M^{\theta}(n,r),\M^{-\theta}(n,r)$ are
sympletomorphic via  the isomorphism induced by $(A,B,i,j)\mapsto (B^*,-A^*, j^*,-i^*)$.

We also consider the varieties $\bar{\M}(n,r),\bar{\M}^\theta(n,r)$ that are obtained similarly
but with the space $R=\operatorname{End}(V)\oplus \Hom(V,W)$ replaced with $\mathfrak{sl}(V)\oplus
\Hom(V,W)$. So $\M(n,r)=\C^2\times \bar{\M}(n,r)$ and $\M^\theta(n,r)=\C^2\times \bar{\M}^\theta(n,r)$.

\subsection{Quantizations}\label{SS_Gies_quant}
Let us discuss some questions about the quantizations $\bar{\A}_\lambda^\theta(n,r)$ of $\bar{\M}^\theta(n,r)$.
These quantizations are defined via
$$\bar{\A}_\lambda^\theta(n,r)=\left[D(\bar{R})/D(\bar{R})\{x_R-\lambda \operatorname{tr}(x)|x\in \g\}|_{(T^*\bar{R})^{\theta-ss}}\right]^G$$

Recall that $\A_\lambda(n,r)=D(\C)\otimes \bar{\A}_\lambda(n,r)$. Let $t$ be the coordinate on $\C$.
Then  the categories $\OCat_\nu(\A_\lambda(n,r))$ and $\OCat_\nu(\bar{\A}_\lambda(n,r))$ are equivalent:
a functor $\OCat_\nu(\bar{\A}_\lambda(n,r))\rightarrow \OCat_\nu(\A_\lambda(n,r))$ is given by
$M\mapsto \C[t]\otimes M$ and a quasi-inverse functor sends $N$ to the annihilator of $\partial_t$.

Now let us discuss periods. By \cite[Section 5]{quant}, the period of $\bar{\A}_\lambda^\theta(n,r)$ is $\lambda-\rho$, where $\varrho$ is half the character of the action of $\GL(n)$ on $\Lambda^{top}R$. So $\varrho=-r/$ and the period
of  $\bar{\A}_\lambda^\theta(n,r)$ is $\lambda+r/2$.

So we have
\begin{equation}\label{eq:opp_iso}
\bar{\A}_{\lambda}(n,r)^{opp}\cong \bar{\A}_{-\lambda-r}(n,r).
\end{equation}

\begin{Lem}\label{Lem:iso}
We have $\bar{\A}_\lambda(n,r)\cong \bar{\A}_{-\lambda-r}(n,r)$.
\end{Lem}
\begin{proof}
Recall, Section \ref{SS_sympl_sing_quant},
that $\bar{\A}_{\lambda}(n,r)\xrightarrow{\sim} \Gamma(\bar{\A}_{\lambda}^{\pm \theta}(n,r))$.
We have produced a symplectomorphism $\iota:\bar{\M}^\theta(n,r)\xrightarrow{\sim}\bar{\M}^{-\theta}(n,r)$ via
$(A,B,i,j)\rightarrow (B^*,-A^*,j^*,-i^*)$.
On the second cohomology this isomorphism induces the multiplication by $-1$.
It follows that $\Gamma(\bar{\A}_\lambda^\theta(n,r))\cong \Gamma(\bar{\A}^{-\theta}_{-\lambda-r}(n,r))$
and our claim follows.
%Recall that  Let us remark that this map is not a symplectomorphism, if
%$\omega^{\theta},\omega^{-\theta}$ are natural symplectic forms on $\M^\theta(n,r), \M^{-\theta}(n,r)$
\end{proof}

We conclude that $\bar{\A}_{\lambda}(n,r)^{opp}\cong \bar{\A}_{\lambda}(n,r)$.

Now let us discuss derived localization and wall-crossing bimodules.  The following is a special
case of the main result of \cite{MN_der}.

\begin{Lem}\label{Lem:MN_der_loc}
Derived localization holds for $(\lambda,\theta)$ if and only if $\bar{\A}_\lambda(n,r)$
has finite homological dimension.
\end{Lem}

\subsection{Fixed point subvarieties and their quantizations}
Recall that the torus $T=\C^\times\times T_0$, where $T_0$ is a maximal torus in $\GL(r)$, acts on $\M^\theta(n,r)$.

Now want to understand the structure of $\M^\theta(n,r)^{\nu_0(\C^\times)}$ for a non-generic $\nu_0$
that lies in the interior of a codimension $1$ face in a chamber. Let $\nu_0(t)=(\operatorname{diag}(t^{d_1},\ldots,
t^{d_r}),t^k)$. The following claim follows directly from the description of the tangent space in the
fixed points given in \cite[Theorem 3.2]{NY}.

\begin{Lem}\label{Lem:1param_subgr_walls}
The walls in the lattice of one-parameter subgroups are given by $k=0$ and $d_i-d_j=sk$ with $s\in \Z$
such that $|s|<n$.
\end{Lem}

The description of $\M^\theta(n,r)^{\nu_0(\C^\times)}$ for $\nu_0$ with $k=0$ is classical.

\begin{Lem}\label{Lem:fixed_pt1}
For $k=0$ (and pairwise different $d_1,\ldots,d_r$), we have
$$\M^\theta(n,r)^{\nu_0(\C^\times)}=\bigsqcup_{n_1+\ldots+n_r=n}\prod_{i=1}^r \M^\theta(n_i,1).$$
\end{Lem}
Recall that $\M^\theta(n_i,1)=\operatorname{Hilb}_{n_i}(\C^2)$ (when $n_i=0$, we assume
that this variety is a single point). Lemma \ref{Lem:fixed_pt1} gives an identification of
$\M^\theta(n,r)^T$ with the set of $r$-multipartitions of $n$. Indeed, the fixed points of
$\C^\times$ in $\M^\theta(n_i,1)$ are in a one-to-one correspondence with the partitions of
$n_i$.

The description for $s\neq 0$ is less classical but is also known. Using an analog of the previous lemma, this description
easily reduces to the case when $r=2$. Here the components of the fixed points are quiver varieties for
finite type A Dynkin diagrams.  We will not need a more precise description in this paper.

Now, for $\nu_0(t)=(\operatorname{diag}(t^{d_1},\ldots, t^{d_r}),1)$ with $d_1\gg d_2\gg\ldots
\ldots\gg d_r$, we want to describe the sheaves $\Ca_{\nu_0}(\A_\lambda^\theta(n,r))$.

\begin{Prop}\label{Prop:Cartan_subquot}
For the irreducible component $Z$ of $\M^\theta(n,r)^{\nu_0(\C^\times)}$ corresponding to
the $r$-multipartition $(n_1,\ldots,n_r)$ of $n$, we have  $\Ca_{\nu_0}(\A_\lambda^\theta(n,r))|_{Z}=\bigotimes_{i=1}^r \A^\theta_{\lambda+i-1}(n_i,1)$ (the $i$th factor is absent if $n_i=0$). Moreover,
for a Zariski generic $\lambda\in \C$, we have $\Ca_{\nu_0}(\A_\lambda(n,r))=
\bigoplus \bigotimes_{i=1}^r \A_{\lambda+(i-1)}(n_i,r)$, where the summation
is taken over all $r$-multipartitions $(n_1,\ldots,n_r)$.
\end{Prop}
\begin{proof}
Let $\mu=(n_1,\ldots,n_r)$, $Z_\mu$ denote the corresponding component of $X^{\nu_0(\C^\times)}$.
Let $Y_\mu$ denote the contracting vector bundle for $Z_\mu$. According to (4) of Proposition \ref{Prop:cat_O},
we need to compute the Chern character of $Y_\mu$.

First, consider the following situation. Set $V:=\C^n, W=\C^r$. Choose a decomposition $W=W^1\oplus W^2$
with $\dim W^i=r_i$ and consider the one-parameter subgroup $\alpha$ of $\GL(W)$ acting trivially on $W^2$ and by $t\mapsto t$ on $W^1$. The components of the  fixed points in $\M^\theta(n,r)$ are in one-to-one correspondence with decompositions
on $n$ into the sum of two parts. Pick such a decomposition $n=n_1+n_2$ and consider the splitting $V=V^1\oplus V^2$
into the sum of two spaces of the corresponding dimensions. Then the corresponding component $Z'\subset \M^\theta(n,r)^{\alpha(\C^\times)}$ is $\M^\theta(n_1,r_1)\times \M^\theta(n_2,r_2)$.

Nakajima has described the contracting bundle $Y\rightarrow Z'$, this is basically in \cite[Proposition 3.13]{Nakajima_tensor}. This is the bundle on $Z'=\M^\theta(n_1,r_1)\times
\M^\theta(n_2,r_2)$ that descends from the $\GL(n_1)\times \GL(n_2)$-module $\ker \beta^{12}/\operatorname{im}\alpha^{12}$,
where $\alpha^{12},\beta^{12}$ are certain $\GL(n_1)\times \GL(n_2)$-equivariant linear maps
$$\Hom(V^2,V^1)\xrightarrow{\alpha^{12}}\Hom(V^2,V^1)^{\oplus 2}\oplus \operatorname{Hom}(W^2,V^1)\oplus \Hom(V^2,W^1)\xrightarrow{\beta^{12}}\Hom(V^2,V^1).
$$ We do not need to know the precise form of the maps $\alpha^{12},\beta^{12}$, what we need is that $\alpha^{12}$ is injective
while $\beta^{12}$ is surjective. So $\ker \beta^{12}/\operatorname{im}\alpha^{12}\cong \operatorname{Hom}(W^2,V^1)\oplus \Hom(V^2,W^1)$,
an isomorphism of $\GL(n_1)\times \GL(n_2)$-modules.

It is easy to see that if $\alpha':\C^\times \rightarrow \GL(r_1)$ is a homomorphism of the form $t\mapsto \operatorname{diag}(t^{d_1},\ldots,t^{d_k})$ with $d_1,\ldots,d_k\gg 0$, then the contracting bundle
for the one-parametric subgroup $(\alpha',1):\C^\times \rightarrow \GL(W^1)\times \GL(W^2)$
coincides with the sum of the contracting bundles for $\alpha'$ and for $\alpha$.
So we see inductively that the vector bundle $Y_\mu$ on $Z_\mu$ descends from the
$\prod_{i=1}^r \GL(n_i)$-module $\sum_{i=1}^r \left((\C^{n_i})^{\oplus r-i}\oplus (\C^{n_i*})^{\oplus i-1}\right)$.
Therefore the 1st Chern class of $Y_\mu$ is $\sum_{i=1}^r (r+1-2i)c_i$, where we write $c_i$
for the generator of $H^2(\M^\theta(n_i,1))$ (that is the first Chern class of $\mathcal{O}(1)$).

Now to prove the claim of this proposition we will use  (4) of Proposition \ref{Prop:cat_O}.
The map $\iota^*_Z$ sends the generator $c$ of $H^2(\M^\theta(n,r))$ to $\sum_{i=1}^r c_i$.
The period of $\A_\lambda^\theta(n,r)$ is $(\lambda+r/2)c$. So the period of $\A_{\lambda+i-1}^\theta(n_i,1)$
is $(\lambda+1/2+i-1)c_i$. Now the statement of this proposition is a direct corollary of
(4) of Proposition \ref{Prop:cat_O}.
\end{proof}

\subsection{Symplectic leaves and slices}\label{SS_sympl_slices}
Here we want to describe the symplectic leaves of $\bar{\M}(n,r)$ and study the
structure of the variety near a symplectic leaf.

\begin{Lem}\label{Lem:0_leaf}
The point $0$ is a single leaf of $\bar{\M}(n,r)$.
\end{Lem}
\begin{proof}
It is enough to show that the maximal ideal of $0$ in $\C[T^*\bar{R}]^G$ is Poisson. Since $\bar{R}$ does not include
the trivial $G$-module as a direct summand, we see that all homogeneous elements in $\C[T^*\bar{R}]^G$ have degree $2$
or higher. It follows that the bracket of any two homogeneous elements also has degree 2 or higher and our claim is proved.
\end{proof}

Now let us describe the slices to symplectic leaves in $\M(n,r)$. When $r=1$, then $\M(n,r)=\C^{2n}/S_n$
and the description is easy.

\begin{Prop}\label{Prop:leaves}
Let $r>1$.  Then the following is true:
\begin{enumerate}
\item
The symplectic leaves of $\bar{\M}(n,r)$ are parameterized by the unordered collections of numbers
$(n_1,\ldots,n_k)$ with $\sum_{i=1}^k n_i\leqslant n$.
\item There is a transversal slice to the leaf as above that is isomorphic to the formal
neighborhood of 0 in $\prod_{i=1}^k \bar{\M}(n_i,r)$.
\end{enumerate}
\end{Prop}
\begin{proof}
Pick $x\in \bar{\M}(n,r)$. We can view $T^*R$ as the representation space of
dimension $(n,1)$ for the double $\overline{Q}^r$ of the quiver $Q^r$ obtained from $Q$ by
adjoining the additional vertex $\infty$ with $r$ arrows from
the vertex $0$ in $Q_0$ to $\infty$. Pick a semisimple representation of $\overline{Q}^r$ lying over $x$. This representation decomposes as $r^0\oplus r^1\otimes U_1\oplus\ldots\oplus r^k\otimes U_k$, where $r^0$ is an irreducible representation of $DQ^w$ with dimension  $(n_0,1)$ and $r^1,\ldots,r^k$ are pairwise nonisomorphic irreducible representations of $\bar{Q}$,  all $r^i, i=1,\ldots,k,$  have dimension $1$. Let $n_i=\dim U_i$.

According to \cite[2.1.6]{BL},  we have a decomposition $\bar{\M}(n,r)^{\wedge_x}\cong D\times \prod_{i=1}^k \bar{\M}(n_i,r)^{\wedge_0}$
of Poisson formal schemes, where $D$ stands for the symplectic formal disk and $\bullet^{\wedge_x}$ indicates
the formal neighborhood of $x$.  From Lemma \ref{Lem:0_leaf} it now follows
that the locus described in the previous is a union of leaves. So in order to prove the entire proposition, it remains
to show that the locus  is irreducible. This easily  follows from \cite[Theorem 1.2]{CB_geom}.
\end{proof}

An isomorphism $\M(n,r)^{\wedge_x}\cong D\times \prod_{i=1}^k \bar{M}(n_i,r)^{\wedge_0}$
can be quantized (see \cite[Section 5.4]{BL}) to
\begin{equation}\label{eq:quant_compl_iso}
\bar{\A}_\lambda(n,r)_\hbar^{\wedge_0}=\mathbf{A}_\hbar^{\wedge_0}
\widehat{\otimes}_{\C[[\hbar]]}\bar{\A}_\lambda(n_1,r)_\hbar^{\wedge_0}\widehat{\otimes}_{\C[[\hbar]]}\ldots
\widehat{\otimes}_{\C[[\hbar]]}\bar{\A}_\lambda(n_k,r)_\hbar^{\wedge_0}.
\end{equation}
Here the notation is as follows. By $\mathbf{A}_\hbar^{\wedge_0}$ we mean the formal Weyl algebra
quantizing $D$. The subscript $\hbar$ means the Rees algebra (with respect to the filtration by order
of differential operators). We note that the map $\hat{r}$ from \cite[Section 5.4]{BL} sends $\lambda$
to $(\lambda,\lambda,\ldots,\lambda)$ hence all parameters in the right hand side of (\ref{eq:quant_compl_iso})
are all equal to $\lambda$.

\subsection{Harish-Chandra bimodules and restriction functors}\label{SS_HC}
Let $\A,\A'$ be two quantizations of the same positively graded Poisson algebra $A$, where, for simplicity,
the degree of the Poisson bracket is $-1$. For example, we can take $A=\C[\bar{\M}(n,r)]$,
$\A=\bar{\A}_\lambda(n,r), \A'=\bar{\A}_{\lambda'}(n,r)$ (with filtrations coming from the
filtration on $D(\bar{R})$ by the order of differential operator).

Let us recall the definition of a HC $\A$-$\A'$-bimodule. By definition, this is
$\A$-$\A'$-bimodule $\B$ that can be equipped with a {\it good filtration}, i.e.,  an $\A$-$\A'$-bimodule
filtration subject to the following two properties:
\begin{itemize}
\item the induced left and right $A$-actions on $\gr\B$ coincide,
\item $\gr\B$ is a finitely generated $A$-module.
\end{itemize}

By a homomorphism of Harish-Chandra bimodules we mean a  bimodule homomorphism. The category of HC $\A$-$\A'$-bimodules is denoted by $\HC(\A\text{-}\A')$. We also consider the full subcategory $D^b_{HC}(\A\text{-}\A')$ of the derived category of  $\A$-$\A'$-bimodules with Harish-Chandra homology.

By the associated variety of a HC bimodule $\B$ (denoted by $\VA(\B)$) we mean the support in $\operatorname{Spec}(A)$ of the coherent sheaf $\gr\B$, where the associated graded is taken with respect to a good filtration. It is easy to see that $\gr\B$ is a Poisson $A$-module so $\VA(\B)$ is the union of symplectic leaves (assuming $\operatorname{Spec}(A)$
has finitely many leaves).

%Using associated varieties and the finiteness of the number of the leaves it is easy to prove the following standard %result.
%
%\begin{Lem}\label{Lem:fin_length}
%Any HC bimodule has finite length.
%\end{Lem}

Let $\A''$ be another filtered quantization of $A$.
For $\B^1\in \HC(\A'\text{-}\A)$ and $\B^2\in \HC(\A''\text{-}\A')$ we can take their tensor product $\B^2\otimes_{\A'}\B^1$. This is easily
seen to be a HC $\A''$-$\A$-bimodule. Also the derived tensor product of the objects from $D^b_{HC}(\A''\text{-}\A'),D^b_{HC}(\A'\text{-}\A)$
lies in $D^-_{HC}(\A''\text{-}\A)$ (and in $D^b_{HC}(\A''\text{-}\A)$ provided $\A'$ has finite homological dimension).

We will need the following result from \cite{B_ineq} that is a part of  \cite[Theorem 1.3]{B_ineq}.

\begin{Lem}\label{Lem:fin_length}
Let $\A$ be the algebra of global section of a quantization of a symplectic resolution.
Then $\A$ has finite length as an $\A$-$\A$-bimodule.
\end{Lem}

Now let us proceed to restriction functors defined in \cite[Section 5.4]{BL} for HC bimodules over
the algebras $\bar{\A}_{\lambda}(n,r)$. Let us write $\bar{\A}_\lambda(n)$ for $\bar{\A}_\lambda(n,r)$.
Pick a point $x$ in the leaf corresponding to the unordered collection $\mu=(n_1,\ldots,n_k)$
and set $\bar{\A}_\lambda(\mu)=\bigotimes_{i=1}^k \bar{\A}_\lambda(n_i)$. Then we have
a functor $$\bullet_{\dagger,\mu}:\operatorname{HC}(\bar{\A}_{\lambda}(n)\text{-}\bar{\A}_{\lambda'}(n))
\rightarrow \operatorname{HC}(\bar{\A}_{\lambda}(\mu)\text{-}\bar{\A}_{\lambda'}(\mu)).$$

We will need several facts about the restriction functors established in
\cite[Section 5.5]{BL}.

\begin{Prop}\label{Prop:dagger_prop}
The following is true.
\begin{enumerate}
\item The functor $\bullet_{\dagger,x}$ is exact and intertwines tensor products
(as well as $\operatorname{Tor}$'s). It also intertwines one-sided Hom's (as well
as $\operatorname{Ext}$'s).
\item
The associated variety $\VA(\B_{\dagger,x})$ is uniquely characterized by  $D\times \VA(\B_{\dagger,\mu})^{\wedge_0}=\VA(\B)^{\wedge_x}$.
\end{enumerate}
\end{Prop}

Now let $\mathcal{L}$ denote the leaf of $x$. Consider the subcategory
$\operatorname{HC}_{\overline{\mathcal{L}}}(\bar{\A}_{\lambda}(n)\text{-}\bar{\A}_{\lambda'}(n))$
of all objects whose associated variety is contained in $\overline{\mathcal{L}}$.
Now suppose that $\mathcal{L}$ corresponds to $\mu$. It follows that
$\bullet_{\dagger,x}$ maps $\operatorname{HC}_{\overline{\mathcal{L}}}(\bar{\A}_{\lambda}(n)\text{-}\bar{\A}_{\lambda'}(n))$
to the category $\operatorname{HC}_{fin}(\bar{\A}_{\lambda}(\mu)\text{-}\bar{\A}_{\lambda'}(\mu))$
of finite dimensional $\bar{\A}_{\lambda}(\mu)\text{-}\bar{\A}_{\lambda'}(\mu)$-bimodules.
As was shown in \cite[Section 3.3]{WC_res}, this functor admits a right adjoint
$$\bullet^{\dagger,x}:\operatorname{HC}_{fin}(\bar{\A}_{\lambda}(\mu)\text{-}\bar{\A}_{\lambda'}(\mu))
\rightarrow \operatorname{HC}_{\overline{\mathcal{L}}}(\bar{\A}_{\lambda}(n)\text{-}\bar{\A}_{\lambda'}(n)).$$

We would like to point out that the results and constructions explained above in this section
generalize to products of the algebras $\bar{\A}_\lambda(?,r)$ in a straightforward way.

Let us finish this section with a further discussion of wall-crossing bimodules. Namely, for
$\lambda\in \C, \chi\in \Z$, we have the $\bar{\A}_{\lambda+\chi}(n,r)$-$\bar{\A}_\lambda(n,r)$-bimodule
$\bar{\A}^0_{\lambda,\chi}(n,r):=[D(\bar{R})/D(\bar{R})\{\xi_R-\lambda \operatorname{tr}\xi|\xi\in \g\}]^{G,\chi}$.
This gives an example of a HC $\bar{\A}_{\lambda+\chi}(n,r)$-$\bar{\A}_\lambda(n,r)$-bimodule.

\begin{Lem}\label{Lem:transl_properties}
The following is true.
\begin{enumerate}
\item If $(\lambda+\chi,\theta)$ satisfies abelian localization, then $\bar{\A}_{\lambda,\chi}^0(n,r)\cong
\bar{\A}_{\lambda,\chi}^{(\theta)}(n,r)$.
\item $(\lambda,\theta)$ satisfies abelian localization if and only if,
for all sufficiently large $n\in \Z$,  the bimodules $\bar{\A}^0_{\lambda+mn\theta, n\theta}(n,r)$
and $\bar{\A}^0_{\lambda+(m+1)n\theta,-n\theta}(n,r)$ are mutually inverse Morita equivalences.
\item If $x\in \overline{\M}(n,r)$ corresponds to a collection $\mu=(n_1,\ldots,n_k)$, then
$\bar{\A}^0_{\lambda,\chi}(n,r)_{\dagger,x}=\bigotimes_{i=1}^k \bar{\A}^0_{\lambda,\chi}(n_i,r)$.
\end{enumerate}
\end{Lem}
\begin{proof}
(1) is \cite[Lemma 5.25]{BL}. (2) follows from \cite[Lemma 5.26]{BL} and \cite[Proposition 5.20]{BPW}.
(3) is a special case of \cite[(5.11)]{BL}.
\end{proof}

\section{Finite dimensional representations and structure of category $\mathcal{O}$}
\subsection{Finite dimensional representations}\label{SS_fin_dim}
The goal of this section is to prove the following statement.

\begin{Prop}\label{Prop:support_local}
The sheaf $\bar{\A}_\lambda^\theta(n,r)$ has a coherent module supported on $\bar{\rho}^{-1}(0)$ if and only if $\lambda=\frac{s}{n}$ with $s$ and $n$ coprime. If that is the case, then the category $\bar{\A}_\lambda^\theta(n,r)\operatorname{-mod}_{\rho^{-1}(0)}$ of all coherent $\bar{\A}_\lambda^\theta(n,r)$-modules
supported on $\rho^{-1}(0)$ is equivalent to $\operatorname{Vect}$.
\end{Prop}

In the proof (and also below) we will need the following lemma.

\begin{Lem}\label{Lem:hw_techn}
Let $\mathcal{C}$ be a highest weight category, where the classes of standard and costandard
objects coincide. Let $R:D^b(\mathcal{C})\xrightarrow{\sim} D^b(\mathcal{C}^{\vee})$
be  a Ringel duality functor. The following conditions are equivalent:
\begin{enumerate}
\item $\mathcal{C}$ is semisimple.
\item We have $H_0(R(L))\neq 0$ for every simple object $L$.
\item every simple  lies in the head of a costandard object.
\end{enumerate}
\end{Lem}
\begin{proof}
Recall that $R$ can be realized as $R\Hom_{\mathcal{C}}(T,\bullet)$, where $T$ stands for the sum of all indecomposable
tiltings in $\mathcal{C}$. The implication (1)$\Rightarrow$(2) is clear. The implication (2)$\Rightarrow$(3) follows from the fact that every costandard object in a highest weight category is a quotient of a tilting.

Let us prove (3)$\Rightarrow$(1). Let $\tau$ be a maximal (with respect to the coarsest highest weight ordering) label. Then the simple $L(\tau)$ lies in the head of some costandard, say $\nabla(\sigma)$. But all simple constituents of $\nabla(\sigma)$ are $L(\xi)$ with $\xi\leqslant \sigma$. It follows that $\sigma=\tau$. Since $L(\tau)$ lies in the head of $\nabla(\tau)$ and also coincides with the socle, we see that $\nabla(\tau)=L(\tau)$. So $L(\tau)$ is injective. Since the classes of standards and costandard coincide, it is also a projective object. So $L(\tau)$ itself forms a block.
 %and therefore spans
%a block in the category.

Since this holds for any maximal $\tau$, we deduce that the category $\mathcal{C}$ is semisimple.
\end{proof}

\begin{proof}[Proof of Proposition \ref{Prop:support_local}]
The proof is in several steps.

{\it Step 1}. Let us take dominant  $\nu$, i.e., $\nu(t)=(\alpha(t),t)$, where $\alpha(t)=\operatorname{diag}(t^{d_1},\ldots, t^{d_r})$ with $d_1\gg d_2\gg\ldots\gg d_r$.
We can assume that $\lambda$ is as in Proposition \ref{Prop:cat_O}.
Let us show that for the category $\mathcal{O}(\bar{\A}_\lambda(n,r))$ condition (2) of Lemma \ref{Lem:hw_techn}
is equivalent to all simples have support of dimension $rn-1$. By Lemmas \ref{Lem:long_wc_perv}
and \ref{Lem:WC_Ringel}, the functor $R$ is a perverse equivalence with respect to the filtration by
dimension of support. So the claims that  all simples have dimension of support equal to $rn-1$
is equivalent to $H_0(RL)\neq 0$ for all simple $L$. By Lemma \ref{Lem:cost_classes}, the classes
of standards and costandards in $K_0(\OCat_\nu(\A^\theta_\lambda))$ coincide and we are done by
Lemma \ref{Lem:hw_techn}.

%This follows from results of Section \ref{SS_long_WC}.
%Also from the results mentioned in that section it follows that every simple in $\mathcal{O}$ has support of
%dimension $rn-1$ provided the algebra $\bar{\A}_\lambda(n,r)$ is simple.

{\it Step 2}. Suppose that $\lambda$ is Zariski generic and  $\bar{\A}_\lambda(n,r)$
has a finite dimensional representation.
Let us prove that the denominator of $\lambda$ is precisely $n$. Indeed, a finite dimensional
module lies in the category $\mathcal{O}_\nu$. It follows that the algebra $\Ca_{\nu_0}(\bar{\A}_\lambda(n,r))$
has a finite dimensional representation, equivalently, there is a simple of GK dimension $1$
in $\OCat_\nu(\Ca_{\nu_0}(\A_\lambda(n,r)))$. Recall that $\Ca_{\nu_0}(\A_\lambda(n,r))=
\bigoplus\bigotimes_{i=1}^r \A_\lambda(n_i,1)$, where the sum is taken over all
$r$-multipartitions $(n_1,\ldots,n_r)$ of $n$. The minimal GK dimension of a module in
the category $\mathcal{O}$ over $\bigotimes_{i=1}^r \A_\lambda(n_i,1)$ is bounded from
below by the number of positive integers among $n_1,\ldots,n_r$. So if there is a finite
dimensional $\bar{\A}_\lambda(n,r)$-module, then there is a GK dimension $1$ module
in the category $\OCat_\nu(\A_\lambda(n,1))$, equivalently, a finite dimensional
$\bar{\A}_\lambda(n,1)$-module. The latter is a spherical rational Cherednik algebra
for $S_n$. As was mentioned in Section \ref{SS_known_results}, it
has a finite dimensional representation only when the denominator of $\lambda$ is $n$.

{\it Step 3}. Let us prove that the algebra $\bar{\A}_\lambda(n,r)$ has no proper ideals of infinite codimension
when $\lambda\not\in \mathbb{Q}$ or the denominator of $\lambda$ is bigger than or equal to  $n$. Indeed, let $\mathcal{I}$ be a proper ideal. Pick $x$ in an open symplectic leaf in $\VA(\bar{\A}_\lambda(n,r)/\mathcal{I})$
and consider the ideal $\mathcal{I}_{\dagger,x}$ in the slice algebra for $x$.
The latter ideal has finite codimension by (2) of Proposition \ref{Prop:dagger_prop}.
By Section \ref{SS_sympl_slices}, this slice algebra is $\bigotimes_{i=1}^k
\bar{\A}_\lambda(n_i,r)$, where $n_1,\ldots,n_k$ are positive integers with $\sum n_i\leqslant n$.
So, by our assumption on $\lambda$ and Step 2, the slice algebra has no finite dimensional
representations. We arrive at a contradiction, which proves the statement in the beginning
of the step.  In particular, we see that the algebra $\bar{\A}_\lambda(n,r)$ is simple
when $\lambda$ is irrational or the denominator of $\lambda$ is bigger than $n$.
We also see that in this case the category $\mathcal{O}$ is semisimple, this follows from
Step 1 and Lemma \ref{Lem:O_holon}.

{\it Step 4}. Now consider the case when $\lambda$ has denominator $n$.
We claim that there is
a finite dimensional module in $\OCat_\nu(\bar{\A}_\lambda(n,r))$. The case when $n=1$
is easy, here $\bar{\A}_\lambda(1,r)=D^\lambda(\mathbb{P}^{r-1})$. So we assume that $n>1$.

 By Lemma \ref{Lem:O_holon}, for any $M\in \mathcal{O}_\nu(\bar{\A}_\lambda(n,r))$, we have  $\operatorname{GK-}\dim M=\frac{1}{2}\dim \VA(\bar{\A}_\lambda(n,r)/\operatorname{Ann}M)$.
Now Step 3 implies that  all simples in $\OCat_\nu$ are either finite dimensional or have support
of dimension $rn-1$. If there are no finite dimensional modules, then
%By Lemma \ref{Lem:top_cohom}, the dimension of the middle homology of $\bar{\M}^\theta(n,r)$ is $1$.
%Thanks to Proposition \ref{Prop:CC_inject}, the number of finite dimensional irreducibles is 0 or $1$. If there is one such %module,
%then the category of finite dimensional modules is semisimple because $\mathcal{O}(\bar{\A}_\lambda(n,r))$ is a highest weight category.
thanks to Step 1, $\mathcal{O}(\bar{\A}_\lambda(n,r))$ is semisimple.

Consider a one-parameter subgroups $\nu_0:t\mapsto (\alpha(t),1)$, it
is in  a face of the chamber of $\nu$. %Proposition \ref{Prop:orders}
%implies that $\alpha\prec^\lambda (\alpha,1)$. Now we can use Lemma \ref{Lem:parab_ind}.
%We remark that, since $d_1\gg d_2\gg\ldots\gg d_r$,  an element $a\in \A$ has a positive weight for $(\alpha,1)$ if and only if it either %has
%a positive weight for $\alpha$ or has zero weight for $\alpha$ and positive weight for $\C^\times$ (the direct
%factor in $\GL(r)\times \C^\times$). This follows from the next lemma.
%
%\begin{Lem}\label{Lem:wts}
%Suppose that $\A$ is a filtered algebra with commutative finitely generated positively graded $\gr \A$.  Further, suppose that $\A$
%is equipped with a filtration preserving action of a torus $T$. Let us define an equivalence relation $\sim$ on $\operatorname{Hom}$
%\end{Lem}
%Let us write $\Delta_{\nu_0},\underline{\Delta}$ for the Verma module functors %$\underline{\Delta}:\C[\M^\theta(n,r)^T]\operatorname{-mod}\rightarrow \mathcal{O}(\Ca_{\nu_0}(\A_\lambda(n,r)))$
%and $\Delta_{\nu_0}:\mathcal{O}(\Ca_{\nu_0}(\A_\lambda(n,r)))\rightarrow \mathcal{O}(\A_\lambda(n,r))$.
%By (6) of Proposition \ref{Prop:cat_O}, we have $\Delta_\nu=\Delta_{\nu_0}\circ \underline{\Delta}$.
The  category $\mathcal{O}_\nu(\Ca_{\nu_0}(\A_\lambda(n,r)))$ is not semisimple, it has the
category $\mathcal{O}$ for the Rational Cherednik algebra with parameter $\lambda$
as a summand. By Proposition \ref{Prop:stand_filt}, $\mathcal{O}_\nu(\A_\lambda(n,r))$
is not semisimple.
%:
%there is a nonzero homomorphism $\varphi:\Delta^0(p_2)\rightarrow \Delta^0(p_1)$, where $p_1=(\varnothing^{r-1}, (n)),
%p_2=(\varnothing^{r-1}, (n-1,1))$. So we get a homomorphism %$\Delta_{\nu_0}(\varphi):\Delta_\nu(p_2)=\Delta_{\nu_0}(\underline{\Delta}(p_2))
%\rightarrow \Delta_{\nu_0}(\underline{\Delta}(p_1))=\Delta_\nu(p_1)$. The highest $\nu_0$-weight components
%of $\Delta_\nu(p_2),\Delta_\nu(p_1)$ coincide with $\underline{\Delta}(p_2),
%\underline{\Delta}(p_1)$, respectively, by the construction. The homomorphism $\underline{\Delta}(p_2)\rightarrow %\underline{\Delta}(p_1)$ induced
%by $\underline{\Delta}(\varphi)$ coincides with $\varphi$. It follows that $\underline{\Delta}(\varphi)\neq 0$. We conclude %that $\mathcal{O}(\A_\lambda(n,r))$ is not semisimple.
Consequently, $\bar{\A}_\lambda(n,r)$ has a finite dimensional
representation.

{\it Step 5}. Let us show that the number of irreducible coherent $\bar{\A}_\lambda^\theta(n,r)$-modules supported
on $\rho^{-1}(0)$ cannot be bigger than $1$. Recall that to a module in the category $\mathcal{O}_\nu$ we can assign
its characteristic cycle that is a formal linear combination of the irreducible components of the contracting locus
of $\nu$. This map is injective, see \cite[Section 6]{BLPW}.
So it is enough to show that there is only one lagrangian irreducible
component in  $\rho^{-1}(0)$. Note that the lagrangian irreducible components of $\rho^{-1}(0)$
give a basis in $H^{2nr-2}(\M^\theta(n,r))$.  According to \cite[Theorem 3.8]{NY}, we have
$$\sum_{i}\dim H^{2i}(\M^\theta(n,r))t^i=\sum_\lambda t^{\sum_{i=1}^r (r|\lambda^{(i)}|-i (\lambda^{(i)t})_1)},$$
where the summation is over the set of the $r$-multipartitions $\lambda=(\lambda^{(1)},\ldots,\lambda^{(r)})$.
The highest power of $t$ in the right hand side is $rn-1$, it occurs for a single $\lambda$, namely, for
$\lambda=((n),\varnothing,\ldots,\varnothing)$. This shows $\dim H^{2nr-2}(\M^\theta(n,r))=1$
and completes the proof of the claim in the beginning of this step.

{\it Step 6}. The previous step completes the proof of all claims of the theorem but the claim that the category of modules supported on $\rho^{-1}(0)$ is semisimple. The latter is an easy consequence of the observation that, in a highest weight category, we have
$\operatorname{Ext}^1(L,L)=0$, for every simple $L$.
\end{proof}

We would like to point out that the argument of Step 4 generalizes
to the denominators less than $n$. So in those cases there are also simple $\bar{\A}_\lambda(n,r)$-modules of support with dimension $<rn-1$.

%We will prove that there is a Zariski open subset of $\hat{\param}$ such that (2) of Theorem \ref{Thm:fin dim}
%holds. Thanks to Proposition \ref{Prop:abel_loc1}, this implies (1) of Theorem \ref{Thm:fin_dim}.

\subsection{Proof of Theorem \ref{Thm:cat_O_easy}}\label{SS_cat_O_thm}
In this subsection we will prove Theorem \ref{Thm:cat_O_easy}. We have already seen (Step 3 of the proof of
Proposition \ref{Prop:support_local}) that if the denominator
is bigger than $n$, then the category $\mathcal{O}$ is semisimple.
The case of denominator $n$ will follow from a more precise statement,
Theorem \ref{Thm:catO_str}.

Let us introduce a certain model category. Let $\Cat_n$ denote the nontrivial block for the category $\mathcal{O}$ for the Rational Cherednik algebra $H_{1/n}(n)$
for the symmetric group $\mathfrak{S}_n$. This is also the category of $B$-equivariant perverse sheaves on
$\mathbb{P}^{n-1}$, where $B$ is a Borel subgroup of $\operatorname{PGL}_n$. Let us summarize some properties of this category.
\begin{itemize}
\item[(i)] Its coarsest highest weight poset is linearly ordered: $p_n<p_{n-1}<\ldots<p_1$.
\item[(ii)] The projective objects $P(p_i)$ for $i>1$ are universal extensions $0\rightarrow \Delta(p_{i-1})\rightarrow
P(p_i)\rightarrow \Delta(p_{i})\rightarrow 0$.
\item[(iii)] The indecomposable tilting objects $T(p_{i-1})$ (with $\Delta(p_{i-1})$ in the socle)
for $i>1$ coincide with $P(p_i)$.
\item[(iv)] The simple objects $L(p_i)$ with $i>1$ appear in the heads of tiltings, while
$\operatorname{RHom}_{\Cat_n}(T,L(p_1))$
is concentrated in homological degree $n$.
\item[(v)] There is a unique simple in the Ringel dual category $\Cat_n^\vee$ that appears in the higher cohomology of $\operatorname{RHom}_{\Cat_n}(T,\bullet)$.
\end{itemize}

\begin{Thm}\label{Thm:catO_str}
Consider a parameter of the form $\lambda=\frac{q}{n}$ with coprime $q,n$. Then  the following is true.
\begin{enumerate}
\item The category $\mathcal{O}_\nu(\bar{\A}_\lambda^\theta(n,r))$ has only one nontrivial block that is equivalent
to $\Cat_{rn}$. This block contains an irreducible representation supported on $\bar{\rho}^{-1}(0)$.
\item Suppose that $\nu$ is dominant.  Then the labels in the non-trivial block of
$\mathcal{O}(\bar{\A}_\lambda^\theta(n,r))$ are hooks $h_{i,d}=(\varnothing,\ldots, (n+1-d, 1^{d-1}),\ldots, \varnothing)$
(where $i$ is the number of the diagram where the hook appears) ordered by $h_{1,n}>h_{1,n-1}>\ldots>h_{1,1}>h_{2,n}>\ldots>h_{2,1}>\ldots>h_{r,1}$.
\end{enumerate}
\end{Thm}
\begin{proof}
The proof is in several steps. We again deal with the realization of our category as $\OCat(\bar{\A}_\lambda(n,r))$, where
$\lambda$ is as in Proposition \ref{Prop:cat_O}.

{\it Step 1}. As we have seen in the proof of Proposition \ref{Prop:support_local},
all simples have maximal dimension of support, except one, let us denote it by $L$, which is finite dimensional.
So all blocks but one consist of modules with support of dimension $rn-1$. So (2) of Lemma \ref{Lem:hw_techn}
holds for these blocks and they are equivalent to $\operatorname{Vect}$.
Let $\mathcal{C}$ denote the  block of $L$. The label of $L$, denote it by $p_{max}$,
is the largest in any highest weight ordering, this follows from the proof of (3)$\Rightarrow$(1)
of Lemma \ref{Lem:hw_techn}. For all other labels $p$ the simple $L(p)$ lies in the socle
of the tilting generator $T$.  In other words an analog of (iv) above holds for $\mathcal{C}$ with $rn$ instead of
$n$. In the subsequent steps we will show that $\mathcal{C}\cong \mathcal{C}_{rn}$.

{\it Step 2}. Let us show that an analog of (v) holds for $\mathcal{C}$.  By Lemma \ref{Lem:long_wc_perv},
the higher homology of $\WC_{\theta\rightarrow -\theta}L$ cannot have support of maximal dimension.
It follows that the higher homology is finite dimensional and so is the direct sums of a single simple in $\mathcal{O}(\bar{\A}_{\lambda^-}(n,r))$. Since $\WC_{\theta\rightarrow -\theta}$ is a  Ringel
duality functor (Lemma \ref{Lem:WC_Ringel}), (v) follows.

{\it Step 3}.  Let us show that, in the coarsest highest weight order, there is a unique minimal label for $\mathcal{C}$, say $p_{min}$. This is equivalent to $\mathcal{C}^\vee$ having a unique maximal label because the orders on $\mathcal{C}$ and $\mathcal{C}^\vee$ are opposite. But $\mathcal{C}^\vee$ is equivalent to the nontrivial block in $\OCat(\bar{\A}_{-r-\lambda}(n,r))$. So we are done by Step 1 (applied to $(\lambda^-,-\theta)$ instead of
$(\lambda,\theta)$) of this proof.

{\it Step 4}. Let us show that (v) implies that any tilting in $\mathcal{C}$ but one is projective.
Let $R^\vee$ denote the Ringel duality equivalence $D^b(\mathcal{C}^\vee)\rightarrow D^b(\mathcal{C})^{opp}$.
We have $\Ext^i_{\mathcal{C}}(T(p),L(p'))=\Hom_{\mathcal{C}}(T(p),L(p')[i])=\Hom_{\mathcal{C}^\vee}((R^\vee)^{-1}L(p')[i],
(R^\vee)^{-1}T(p))$. The objects $(R^{\vee})^{-1}T(p)$ are injective so $\Ext^i(T(p),L(p'))=\Hom(H^i((R^\vee)^{-1} L(p')),
(R^\vee)^{-1} T(p))$. Similarly to the previous step (applied to  $\mathcal{C}^\vee$ instead of $\mathcal{C}$ and $(R^{\vee})^{-1}$ instead of $R$),
there is a unique indecomposable injective $I^\vee(p^\vee)$ in $\Cat^\vee$ that admits nonzero maps from a higher
cohomology of $(R^{\vee})^{-1} L(p)$. So if $(R^\vee)^{-1} T(p)\neq I^\vee(p^\vee)$, then $T(p)$ is projective.

{\it Step 5}. We remark that $\Delta(p_{max})$ is projective but not tilting, while $\Delta(p_{min})$ is tilting
but not projective. So the projectives in $\mathcal{C}$ are $\Delta(p_{max})$ and $T(p)$ for $p\neq p_{min}$.
Similarly, the tiltings are $P(p), p\neq p_{max}$, and $\Delta(p_{min})$.

{\it Step 6}. Let $\Lambda$ denote the highest weight poset for $\mathcal{C}$.
Let us define a map $\nu: \Lambda\setminus \{p_{min}\}\rightarrow \Lambda\setminus \{p_{max}\}$.
It follows from  Step 5 that  the head of any tilting in $\mathcal{C}$ is simple. By definition,
$\nu(p)$ is such that $L(\nu(p))$
is the head of $T(p)$. We remark that $\nu(p)< p$ for any highest weight order, this follows from Step 4.

{\it Step 7}. Let us show that any element $p\in \Lambda$ has the form $\nu^i(p_{max})$. Assume the converse
and let us pick the maximal element not of this form, say $p'$. Since $p'\neq p_{max}$, we see that $L(p')$ lies
in the head of some tilting. But the head of any indecomposable tilting is simple by Step 5.
So $\Delta(p')$ is a top term of a filtration   with standard subsequent quotients. By the definition of $\nu$ and the choice of $p'$, $\Delta(p')$ is tilting itself. Any indecomposable tilting but $\Delta(p_{min})$ is projective and we cannot have a standard that is projective and tilting simultaneously in a nontrivial block. So $p'=p_{min}$. But let us pick a minimal element $p''$ in $\Lambda\setminus \{p_{min}\}$. By Step 6, $\nu(p'')<p''$. So $\nu(p'')=p_{min}$. The claim in the beginning of the step is established. This proves (i) for $\mathcal{C}$.

{\it Step 8}. (ii) for $\mathcal{C}$ follows from Step 7 and (iii) follows from (ii) and Step 5.

{\it Step 9}. Let us show that $\#\Lambda=rn$. The minimal projective resolution for $\Delta(p_{min})$
has length $\#\Lambda$, all projectives there are different, and the last term is $\Delta(p_{max})$. It follows that
$\operatorname{RHom}(\Delta(p_{min}),L(p_{max}))$ is concentrated in homological degree $\#\Lambda-1$.
The other tiltings are projectives and $\operatorname{RHom}$'s with them amount to $\Hom$'s. Since $\operatorname{RHom}(T,L(p_{max}))$ is concentrated in homological degree $rn-1$ (this follows from
Lemmas \ref{Lem:long_wc_perv} and \ref{Lem:WC_Ringel}), we are done.

{\it Step 10}. Let us complete the proof of (1). Let us order the labels in $\Lambda$  decreasingly, $p_1>\ldots>p_{rn}$.
Using (ii) we get the following claims.
\begin{itemize}
\item $\operatorname{End}(P(p_i))=\C[x]/(x^2)$ for $i>1$ and $\operatorname{End}(P(p_1))=\C$.
\item $\operatorname{Hom}(P(p_i),P(p_j))$ is 1-dimensional if
$|i-j|=1$ and is $0$ if $|i-j|>1$.
\end{itemize}
Choose some basis elements $a_{i,i+1}, i=1,\ldots,rn-1$  in $\operatorname{Hom}(P(p_{i+1}), P(p_{i}))$
and also basis elements $a_{i+1,i}\in \operatorname{Hom}(P(p_{i}),P(p_{i+1}))$. We remark that
the image of the composition map $\Hom(P(p_{i+1}),P(p_{i}))\times \Hom(P(p_{i}),P(p_{i+1}))
\rightarrow \operatorname{End}(P(p_i))$ spans the maximal ideal. Choose  generators
$a_{ii}$ in the maximal ideals of $\End(P(p_i)), i=2,\ldots,rn$.
Normalize $a_{21}$ by requiring that $a_{21}a_{12}=a_{22}$, automatically, $a_{12}a_{21}=0$. Normalize $a_{32}$
by $a_{23}a_{32}=a_{22}$ and then normalize $a_{33}$ by $a_{33}=a_{32}a_{23}$. We continue normalizing $a_{i+1,i}$
and $a_{i+1,i+1}$ in this way. We then recover the multiplication table in $\operatorname{End}(\bigoplus P(p_i))$
in a unique way. This completes the proof of (1).

{\it Step 11}. Now let us prove (2).
Let us check that the labeling set  $\Lambda$ for the nontrivial block of $\OCat(\bar{\A}^\theta_\lambda(n,r))$ consists of hooks.
For this, it is enough to check that $\Delta(h_{i,d})$ does not form a block on itself.
This is done similarly to  Step 4 in the proof of Theorem \ref{Thm:fin dim}.
Now, according to \cite{Korb},
the hooks are ordered as specified in (2) with respect to the geometric order
on the torus fixed points in $\M^\theta(n,r)$ (note that the sign conventions here and in \cite{Korb}
are different).
%If $d\neq 1$, we take $\mu=\lambda(i+1,j)$. In this case we have a nontrivial
%homomorphism $\Delta(\mu)\rightarrow \Delta(\lambda)$. If $i=1$, take $\mu=\lambda(2,j)$, here we have a nontrivial
%homomorphism $\Delta(\lambda)\rightarrow \Delta(\mu)$. The existence of a nontrivial homomorphism is proved in
%the same way as in the proof of Proposition \ref{Prop:O_ss_criterium} because we have nontrivial homomorphisms
%between $\Delta^0(\lambda),\Delta^0(\mu)$.
\end{proof}

\begin{Rem}\label{Rem:simpl_label}
We can determine the label of the simple supported on $\bar{\rho}^{-1}(0)$ in the category $\OCat$ corresponding to
an arbitrary generic torus. Namely, note that $\bar{\rho}^{-1}(0)$ coincides with the closure of a single contracting component
and that contracting component corresponds to the maximal point. Now we can use results of \cite{Korb} to find a label
of the point: it always has only one nontrivial partition and this partition is either $(n)$ or $(1^n)$.
\end{Rem}

\section{Localization theorems}
In this section we prove Theorem \ref{Thm:loc}. The proof is in the following steps.
\begin{itemize}
\item We apply results of McGerty and Nevins, \cite{MN_ab}, to show that, first, if the abelian localization  fails for $(\lambda,\theta)$, then $\lambda$ is a rational number with denominator not exceeding $n$, and, second,  the parameters $\lambda=\frac{q}{m}$
with $m\leqslant n$ and $-r<\lambda<0$ are indeed singular and the functor $\Gamma_\lambda^\theta$ is
exact when $\lambda>-r, \theta>0$ or $\lambda<0,\theta<0$. Thanks to an isomorphism $\A_{\lambda}^\theta(n,r)\cong
\A_{-\lambda-r}^{-\theta}(n,r)$ (see Sections \ref{SS_sympl_sing_quant} and \ref{SS_Gies_quant}), this reduces the conjecture to checking that the abelian localization
holds for $\lambda=\frac{q}{m}$ with $q\geqslant 0, m\leqslant n$ and $\theta>0$.
\item Then we reduce the proof to the case when the denominator is precisely $n$ and $\lambda,\theta>0$.
\item Then we will study a connection between the algebras $\Ca_{\nu_0}(\bar{\A}_{\lambda}(n,r)),
\Gamma(\Ca_{\nu_0}(\bar{\A}^\theta_{\lambda}(n,r)))$. We will show that the numbers of simples in the categories
$\OCat$ for these algebras coincide.  We deduce the localization theorem from there.
\end{itemize}

The last step is a crucial one and it does not generalize to other quiver varieties.
\subsection{Results of McGerty and Nevins and consequences}\label{SS_MN_appl}
In \cite{MN_ab}, McGerty and Nevins found a sufficient condition for the functor $\Gamma_\lambda^\theta:\A_\lambda^\theta(n,r)\operatorname{-mod}\rightarrow
\A_\lambda(n,r)\operatorname{-mod}$ to be exact (they were dealing with more general
Hamiltonian reductions but we will only need the Gieseker case). Let us explain what their
result gives in the case of interest for us. Consider the quotient functors
$\pi_\lambda: D_R\operatorname{-mod}^{G,\lambda}\twoheadrightarrow \A_\lambda(n,r)\operatorname{-mod}$
and $\pi_\lambda^\theta:D_R\operatorname{-mod}^{G,\lambda}\twoheadrightarrow \A_\lambda^\theta(n,r)\operatorname{-mod}$.
Here $D_R$ stands for the sheaf of differential operators viewed as a microlocal sheaf on $T^*R$,
and $D_R\operatorname{-mod}^{G,\lambda}$ is the category of $(G,\lambda)$-twisted equivariant
D-modules on $R$. The functors $\pi_\lambda,\pi_\lambda^\theta$ are discussed in
\cite[Section 2.3]{BL} or \cite[Section 5.5]{BPW}.

\begin{Prop}\label{Prop:MN}
The inclusion $\ker \pi_{\lambda}^{\det}\subset \ker \pi_\lambda$ holds  provided $\lambda>-r$. Similarly, $\ker\pi_{\lambda}^{\det^{-1}}
\subset \pi_{\lambda}$  provided $\lambda< 0$.
\end{Prop}

I would like to thank Dmitry Korb for explaining me the required modifications to \cite[Section 8]{MN_ab}.

\begin{proof}
We will consider the case $\theta=\det$, the opposite case follows from $\A_\lambda^{\theta}(n,r)\cong \A_{-r-\lambda}^{-\theta}(n,r)$.
The proof closely follows \cite[Section 8]{MN_ab}, where the case of $r=1$ is considered. Instead of $R=\operatorname{End}(V)\oplus
\operatorname{Hom}(V,W)$ they use $R'=\operatorname{End}(V)\oplus \operatorname{Hom}(W,V)$, then, thanks to the partial
Fourier transform, we have $D(R)\operatorname{-mod}^{G,\lambda}\cong D(R')\operatorname{-mod}^{G,\lambda+r}$. The set of weights in $R'$ for a maximal torus $H\subset \GL(V)$
is independent of $r$ so we have the same Kempf-Ness subgroups as in the case $r=1$:  it is enough to consider the subgroups
$\beta$ with tangent vectors (in the notation of \cite[Section 8]{MN_ab}) $e_1+\ldots+e_k$. The shift in {\it loc.cit.} becomes
$\frac{rk}{2}$ (in the computation of {\it loc.cit.} we need to take the second summand $r$ times, that is all that changes). So we get
that $\ker \pi_{\lambda}^{\det}\subset \ker \pi_\lambda$ provided $k(-\frac{r}{2}-\lambda)\not\in \frac{rk}{2}+\Z_{\geqslant 0}$
for all possible $k$ meaning $1\leqslant k\leqslant n$ (the number $-\frac{r}{2}-\lambda$ is $c'$ in {\it loc.cit.}). The condition
simplifies to $\lambda\not\in -r-\frac{1}{k}\Z_{\geqslant 0}$. This implies the claim of the proposition.
\end{proof}

\subsection{Reduction to denominator $n$ and singular parameters}\label{SS_loc_red_to_n}
Proposition \ref{Prop:MN} allows us to show that certain parameters are singular (meaning that the homological
dimension of $\A_\lambda(n,r)$ is infinite).
\begin{Cor}\label{Cor:sing}
The parameters $\lambda$ with denominator $\leqslant n$ and $-r<\lambda<0$ are singular.
\end{Cor}
\begin{proof}
Assume the contrary. By Lemma \ref{Lem:MN_der_loc}, the functors
$R\Gamma_\lambda^{\pm \theta}$ are equivalences. Since the functors $\Gamma_\lambda^{\pm \theta}$
are exact, we see that $\Gamma_{\lambda}^{\pm \theta}$ are equivalences of abelian categories. From the
inclusions $\ker \pi_{\lambda}^{\pm \theta}\subset \ker\pi_\lambda$, we deduce that the functors
$\pi_{\lambda}^{\pm \theta}$ are isomorphic. So the wall-crossing functor $\mathfrak{WC}_{-\theta\leftarrow \theta}=\pi_{\lambda^-}^{-\theta}\circ (\C_{\lambda^--\lambda}\otimes\bullet)\circ L\pi_{\lambda}^{\theta!}$
(see \cite[(2.10)]{BL} for the equality)
is an equivalence of abelian categories (where we modify $\lambda$ by adding a sufficiently large integer).
However, we have seen in the end of Section \ref{SS_fin_dim}, the category $\mathcal{O}_\nu(\A_\lambda)$
is not semisimple. Combining Lemma \ref{Lem:hw_techn} with Lemma \ref{Lem:WC_Ringel}, we see that $\WC_{-\theta\leftarrow \theta}$ cannot be an abelian equivalence.
\end{proof}

Now let us observe that it is enough to check that the abelian localization holds for $\lambda\geqslant 0$
and $\theta>0$. This follows from an isomorphism $\A_{\lambda}^\theta(n,r)\cong \A_{-\lambda-r}^{-\theta}(n,r)$.
This an isomorphism of sheaves on $\M^\theta(n,r)\cong \M^{-\theta}(n,r)$ (see the proof of Lemma \ref{Lem:iso}).

\begin{Lem}\label{Lem:ab_loc_redn}
Suppose that, for all $n$, abelian localization holds for $\A_\lambda^\theta(n,r)$  if $\lambda>0$
and the denominator of $\lambda$ is exactly $n$. Then abelian localization holds for all $\lambda$.
\end{Lem}
\begin{proof}
Let the denominator $n'$ of $\lambda$ be less then $n$. By Lemma \ref{Lem:transl_properties},
the abelian localization holds for $(\lambda,\theta>0)$ if and only if the bimodules
$\A^0_{\lambda+m\chi,\chi}(n,r), \A^0_{\lambda+(m+1)\chi,-\chi}(n,r)$ with
some $\chi\in \Z_{>0}$ define mutually dual Morita equivalences, equivalently, the natural homomorphisms
\begin{equation}\label{eq:nat_homs}
\begin{split}&\A^0_{\lambda+m\chi,\chi}(n,r)\otimes_{\A_{\lambda+m\chi}(n,r)} \A^0_{\lambda+(m+1)\chi,-\chi}(n,r)\rightarrow \A_{\lambda+(m+1)\chi}(n,r),\\
&\A^0_{\lambda+(m+1)\chi,-\chi}(n,r)\otimes_{\A_{\lambda+(m+1)\chi}(n,r)}\A^0_{\lambda+m\chi,\chi}(n,r)\rightarrow \A_{\lambda+m\chi}(n,r)
\end{split}
\end{equation}
are isomorphisms.

Assume the converse: there is $\lambda>0$ such that abelian localization does not hold for
$\A_\lambda^\theta(n,r)$. Let $K^1,C^1,K^2,C^2$ denote the kernel and the cokernel of the first and of the
second homomorphism (for some $m$), respectively. If one of these bimodules is nontrivial, then we can find $x\in \M(n,r)$
such that $K^i_{\dagger,x},C^i_{\dagger,x}$ are finite dimensional and at least one of these bimodules
is nonzero. From the classification of finite dimensional irreducibles, we see that the slice algebras
must be of the form $\bar{\A}_{\lambda}(n',r)^{\otimes k}$. But, by Lemma \ref{Lem:transl_properties}, $\A^0_{\lambda+(m+1)\chi,-\chi}(n,r)_{\dagger,x}=
\bar{\A}^0_{\lambda+(m+1)\chi,-\chi}(n',r)^{\otimes k}, \A^0_{\lambda+m\chi,\chi}(n,r)_{\dagger,x}=
\bar{\A}^0_{\lambda+m\chi,\chi}(n',r)^{\otimes k}$. Further, applying $\bullet_{\dagger,x}$ to (\ref{eq:nat_homs}) we again get  natural homomorphisms. But the localization theorem holds for the algebra $\bar{\A}_{\lambda}(n',r)$ thanks to our inductive assumption, so the homomorphisms of the $\bar{\A}_\lambda(n',r)^{\otimes k}$-bimodules are isomorphisms. This contradiction finishes the proof of the lemma.
\end{proof}

\subsection{Number of simples in $\mathcal{O}(\A_\lambda(n,r))$}\label{SS_loc_simpl_numb}
So we need to prove that the localization theorem holds for positive parameters $\lambda$ with denominator $n$
(the case $\lambda=0$ occurs only if $n=1$ and in that case this is a classical localization theorem for
differential operators on projective spaces).
We will derive the proof from the claim that the number of simple objects in the categories $\OCat_\nu(\bar{\A}_\lambda(n,r))$
and $\OCat_\nu(\bar{\A}_\lambda^\theta(n,r))$ is the same. For this we will need to study the natural homomorphism
$$\varphi:\Ca_{\nu_0}(\bar{\A}_\lambda(n,r))\rightarrow \Gamma(\Ca_{\nu_0}(\bar{\A}^\theta_\lambda(n,r))).$$ Here, as before,
$\nu(t)=(\alpha(t),t), \nu_0(t)=(\alpha(t),1)$,  where $\alpha:\C^\times \rightarrow \GL(r)$ is of the form $t\mapsto (t^{d_1},\ldots,t^{d_r})$, and $d_1\gg d_2\gg\ldots\gg d_r$.

Recall, see Proposition \ref{Prop:Cartan_subquot}, that $\Gamma(\Ca_{\nu_0}(\bar{\A}^\theta_\lambda(n,r)))=\bigoplus \bar{\A}_\lambda(n_1,\ldots,n_r;r)$, where the summation
is taken over all compositions $n=n_1+\ldots+n_r$ and  $\bar{\A}_\lambda(n_1,\ldots,n_r;r)\otimes D(\C)=\bigotimes_{i=1}^r \A_{\lambda+i-1}(n_i,1)$ (the factor $D(\C)$ is embedded into the right hand side ``diagonally'').
Let $\mathcal{B}$ denote the maximal
finite dimensional quotient of $\Gamma(\Ca_{\nu_0}(\bar{\A}_\lambda^\theta(n,r)))$.

\begin{Prop}\label{Prop:surject}
The composition of $\varphi$ with the projection $\Gamma(\Ca_{\nu_0}(\bar{\A}^\theta_\lambda(n,r)))\twoheadrightarrow \mathcal{B}$
is surjective.
\end{Prop}
\begin{proof}
The proof is in several steps.

{\it Step 1}. We claim that it is sufficient to prove that the composition $\varphi_i$ of $\varphi$ with the projection
$\Gamma(\Ca_{\nu_0}(\bar{\A}_\lambda^\theta(n,r)))\rightarrow \bar{\A}_{\lambda+i}(n,1)$ is surjective. Indeed, each $\bar{\A}_{\lambda+i}(n,1),
i=0,\ldots,r-1$ has a unique finite dimensional representation. The dimensions of these representations are
pairwise different, see \cite{BEG}. Namely, if $\lambda=\frac{q}{n}$, then
the dimension is $\frac{(q+n-1)!}{q!n!}$. So $\mathcal{B}$ is the sum of $r$
pairwise non-isomorphic matrix algebras. Therefore the surjectivity of the homomorphism $\Ca_{\nu_0}(\bar{\A}_\lambda(n,r))\rightarrow \mathcal{B}$ follows from the surjectivity of all its $r$ components. We remark that the other summands of
$\Ca_{\nu_0}(\bar{\A}_\lambda(n,r))$ have no finite dimensional representations.

{\it Step 2}. Generators of $\bar{\A}_{\lambda+i}(n,1)$ are known. Namely, recall that $\bar{\A}_{\lambda+i}(n,1)$ is the spherical subalgebra in the Cherednik algebra $H_c(n)$ for the reflection representation $\h$ of $\mathfrak{S}_n$ with $c=\lambda+i$. The algebra $H_c(n)$ is generated by $\h,\h^*$. Then the algebra $eH_c(n)e$ is generated by $S(\h)^W,S(\h^*)^W$,
see, e.g., the proof \cite[Proposition 4.9]{EG}. On the level of quantum Hamiltonian reduction,
$S(\h)^W$ coincides with the image of $S(\bar{\g})^G\subset D(\bar{\g}\oplus \C^{n*})^G$, while
$S(\h^*)^W$ coincides with the image of $S(\bar{\g}^*)^G$. Here and below we write $\bar{\g}$ for $\mathfrak{sl}_n$. We will show that the images of $S(\bar{\g})^G, S(\bar{\g}^*)^G$ lie in the image of $\varphi_i:\Ca_{\nu_0}(\bar{\A}_{\lambda}(n,r))\rightarrow \bar{\A}_{\lambda+i}(n,1)$, this will establish the surjectivity in Step 1.

{\it Step 3}. Let us produce a natural homomorphism $S(\bar{\g}^*)^G\rightarrow \Ca_{\nu_0}(\bar{\A}_{\lambda}(n,r))$. First of all, recall that
$\bar{\A}_{\lambda}(n,r)$ is a quotient of $D(\bar{\g}\oplus (\C^{*n})^r)^{G}$.
The algebra $S(\bar{\g}^*)^G$ is included into $D(\bar{\g}\oplus (\C^{*n})^{\oplus r})^G$
as the algebra of invariant functions on $\bar{\g}$. So we get a homomorphism $S(\bar{\g}^*)^G\rightarrow
\bar{\A}_{\lambda}(n,r)$. Since the $\C^\times$-action $\nu_0$
used to form $\Ca_{\nu_0}(\bar{\A}_\lambda(n,r))$ is nontrivial only on $(\C^{*n})^{\oplus r}$, we see that
the image of $S(\bar{\g}^*)^G$ lies in $\bar{\A}_{\lambda}(n,r)^{\nu_0(\C^\times)}$. So  we get a
homomorphism $\iota: S(\bar{\g}^*)^G\rightarrow \Ca_{\nu_0}(\bar{\A}_\lambda(n,r))$.

{\it Step 4}. We claim that $\varphi_i\circ \iota$ coincides with the inclusion $S(\bar{\g}^*)^G\rightarrow \bar{\A}_{\lambda+i}(n,1)$. We can filter  the algebra $D(\bar{\g}\oplus (\C^{*n})^{\oplus r})$ by the order of a differential operator. This induces filtrations on $\bar{\A}_{\lambda}(n,r),\bar{\A}^\theta_\lambda(n,r)$. We have similar filtrations on the algebras $\bar{\A}_{\lambda+i}(n,1)$. The filtrations on $\bar{\A}_\lambda(n,r),\bar{\A}_\lambda^\theta(n,r)$ are preserved by $\nu_0$ and hence we have filtrations on $\Ca_{\nu_0}(\bar{\A}_\lambda(n,r)),\Gamma(\Ca_{\nu_0}(\bar{\A}_\lambda^\theta(n,r)))$.
It is clear from the construction of the projection $\Gamma(\Ca_{\nu_0}(\bar{\A}_\lambda^\theta(n,r)))\rightarrow
\bar{\A}_{\lambda+i}(n,1)$ that it is compatible with the filtrations. On the other hand, the images of $S(\bar{\g}^*)^G$ in both $\Ca_{\nu_0}(\bar{\A}_{\lambda}(n,r)),\bar{\A}_{\lambda+i}(n,1)$ lies in the filtration degree 0.
So it is enough to prove the coincidence of the  homomorphisms in the beginning of the step  after passing to associate
graded algebras.

{\it Step 5}. The associated graded homomorphisms coincide with analogous homomorphisms defined on the classical level.
The components of $\M^\theta(n,r)^{\nu_0(\C^\times)}$ that are Hilbert schemes are realized as follows.
Pick an eigenbasis $w_1,\ldots,w_r\in \C^r$ for $T_0$. Then the  Hilbert scheme component $Z_i$ of $\M^\theta(n,r)^{T_0}$
corresponding to the composition $(0^{i-1},n,0^{r-i-1})$
consists of the $G$-orbits of $(A,B,0,j)$, where $j:\C^n\rightarrow \C^r$
is a map with image in  $\C w_j$. In particular, the homomorphism $S(\bar{\g}^*)^G\rightarrow \gr \A_{\lambda+i}(n,1)$
is dual to the morphism given by $(A,B,0,j)\rightarrow A$.

On the other hand, $Z_i$   maps
to $\M(r,n)\quo \nu_0(\C^\times)$ ($Z_i\hookrightarrow \M^\theta(n,r)\twoheadrightarrow
\M(n,r)\twoheadrightarrow \M(n,r)\quo \nu_0(\C^\times)$). The corresponding homomorphism of algebras is the associated graded of $\bar{\A}_\lambda(n,r)^{\nu_0(\C^\times)}\rightarrow \bar{\A}_{\lambda+i}(n,1)$. Then we have the morphism
$\M(r,n)\quo \nu_0(\C^\times)\rightarrow \bar{\g}\quo G$ induced by $(A,B,0,j)\mapsto A$.
The corresponding homomorphism of algebras is the associated graded
of $S(\bar{\g}^*)^G\rightarrow \bar{\A}_\lambda(n,r)^{\nu_0(\C^\times)}$. So we have checked that the associated graded
homomorphism of $\varphi_i\circ\iota:S(\bar{\g}^*)^G\rightarrow \bar{\A}_{\lambda+i}(n,1)$ coincides with that of the embedding
$S(\bar{\g}^*)^G\rightarrow \bar{\A}_{\lambda+i}(n,1)$. This proves the claim of Step 4.

{\it Step 6}. The coincidence of similar homomorphisms $S(\bar{\g})^G\rightarrow \bar{\A}_{\lambda+i}(n,1)$ is established
analogously. The proof of the surjectivity of $\Ca_{\nu_0}(\bar{\A}_{\lambda}(n,r))\rightarrow \bar{\A}_{\lambda+i}(n,1)$ is now complete.
\end{proof}

We still have a Hamiltonian action of $\C^\times$ on $\Ca_{\nu_0}(\bar{\A}_\lambda(n,r))$ (via $\nu$) that makes the homomorphism
$\Ca_{\nu_0}(\bar{\A}_\lambda(n,r))\rightarrow \Gamma(\Ca_{\nu_0}(\bar{\A}_\lambda^\theta(n,r)))$ equivariant. So we can form
the category $\mathcal{O}(\Ca_{\nu_0}(\bar{\A}_\lambda(n,r)))$ for this action. %By Lemma \ref{Lem:prec_spec}, we have
%$\alpha\prec^\lambda (m\alpha,1)$ for $m\gg 0$. We rescale $\alpha$ and assume that $m=1$. Recall, Lemma %\ref{Lem:parab_ind},
By \cite[Section 5.5]{CWR}, we have an isomorphism $\Ca_{\nu}(\Ca_{\nu_0}(\bar{\A}_\lambda(n,r)))\cong
\Ca_{\nu}(\bar{\A}_\lambda(n,r))$. So there is
a natural bijection between the sets of simples in $\mathcal{O}_\nu(\Ca_{\nu_0}(\bar{\A}_\lambda(n,r)))$ and in
$\mathcal{O}_\nu(\bar{\A}_\lambda(n,r))$.

\begin{Prop}\label{Prop:simple_numbers}
The number of simples in $\mathcal{O}(\Ca_{\nu_0}(\bar{\A}_\lambda(n,r)))$
is bigger than or equal to that in $\mathcal{O}(\Gamma(\Ca_{\nu_0}(\bar{\A}_\lambda^\theta(n,r))))$.
\end{Prop}
\begin{proof}
The proof is again in several steps.

{\it Step 1}.
We have a natural homomorphism $\C[\bar{\g}]^G\rightarrow \bigoplus \bar{\A}_\lambda(n_1,\ldots,n_r;r)$. It can be described as follows.
We have an identification $\C[\bar{\g}]^G\cong \C[\h]^{\mathfrak{S}_n}$. This algebra embeds into $\bar{\A}_\lambda(n_1,\ldots,n_r;r)$
(that is a spherical Cherednik algebra for the group $\prod_{i=1}^r \mathfrak{S}_{n_i}$ acting on $\h$) via the inclusion $\C[\h]^{\mathfrak{S}_n}\subset
\C[\h]^{\mathfrak{S}_{n_1}\times\ldots\times \mathfrak{S}_{n_r}}$. For the homomorphism
$\C[\bar{\g}]^G\rightarrow \bigoplus \bar{\A}_\lambda(n_1,\ldots,n_r;r)$ we take the direct sum of these embeddings. Similarly
to Steps 4,5 of the  proof of Proposition \ref{Prop:surject}, the maps $\C[\bar{\g}]^G\rightarrow \Ca_{\nu_0}(\bar{\A}_\lambda(n,r)), \Gamma(\Ca_{\nu_0}(\bar{\A}_\lambda^\theta(n,r)))$
are intertwined by the homomorphism $\Ca_{\nu_0}(\bar{\A}_\lambda(n,r))\rightarrow \Gamma(\Ca_{\nu_0}(\bar{\A}_\lambda^\theta(n,r)))$.

{\it Step 2}. Let $\delta\in \C[\bar{\g}]^G$ be the discriminant. We claim that $\Ca_{\nu_0}(\bar{\A}_\lambda^\theta(n,r))[\delta^{-1}]\xrightarrow{\sim} \Gamma(\Ca_{\nu_0}(\bar{\A}_\lambda^\theta(n,r)))[\delta^{-1}]$. Since $\delta$ is ${\nu_0}(\C^\times)$-stable,
we have $\Ca_{\nu_0}(\bar{\A}_\lambda(n,r))[\delta^{-1}]=\Ca_{\nu_0}(\bar{\A}_\lambda(n,r)[\delta^{-1}])$.
We will describe the algebra $\Ca_{\nu_0}(\bar{\A}_\lambda(n,r)[\delta^{-1}])$ explicitly and see that
$\Ca_{\nu_0}(\bar{\A}_\lambda(n,r)[\delta^{-1}])\xrightarrow{\sim} \Gamma(\Ca_{\nu_0}(\bar{\A}_\lambda^\theta(n,r)))[\delta^{-1}]$.

{\it Step 3.} We start with the description of $\bar{\A}_\lambda(n,r)[\delta^{-1}]$. Let $\bar{\g}^{reg}$ denote the locus
of the regular semisimple elements in $\bar{\g}$. Then $\bar{\A}_\lambda(n,r)[\delta^{-1}]=D(\bar{\g}^{reg}\times \operatorname{Hom}(\C^n,\C^r))\red_\lambda G$. Here $\red_\lambda$ denotes the quantum Hamiltonian reduction with parameter $\lambda$.

Recall that $\bar{\g}^{reg}=G\times_{N_G(\h)}\h^{reg}$ and so $\bar{\g}^{reg}\times \Hom(\C^n,\C^r)=G\times_{N_G(\h)}(\h^{reg}\times \operatorname{Hom}(\C^n,\C^r))$. It follows that
\begin{align*}&D(\bar{\g}^{reg}\times \operatorname{Hom}(\C^n,\C^r))\red_\lambda G=D(\h^{reg}\times \Hom(\C^n,\C^r))\red_\lambda N_G(\h)=\\
&(D(\h^{reg})\otimes D(\operatorname{Hom}(\C^n,\C^r))\red_\lambda H)^{\mathfrak{S}_n}= \left(D(\h^{reg})\otimes D^\lambda(\mathbb{P}^{r-1})^{\otimes n}\right)^{\mathfrak{S}_n}.\end{align*}
Here, in the second line, we write $H$ for the Cartan subgroup of $G=\operatorname{GL}_n(\C)$
and take the diagonal action of $\mathfrak{S}_n$.
In the last expression, it permutes the tensor factors.
A similar argument shows that $\bar{\M}^\theta(n,r)_{\delta}= (T^*(\h^{reg})\times T^*(\mathbb{P}^{r-1})^{n})/\mathfrak{S}_n$
and the restriction of $\bar{\A}^\theta_\lambda(n,r)$ to this open subset is $\left(D_{\h^{reg}}\otimes (D^\lambda_{\mathbb{P}^{r-1}})^{\otimes n}\right)^{\mathfrak{S}_n}$.

{\it Step 4}. Now we are going to describe the algebra $\Ca_{\nu_0}(\left(D(\h^{reg})\otimes D^\lambda(\mathbb{P}^{r-1})^{\otimes n}\right)^{\mathfrak{S}_n})$. First of all, we claim that
\begin{equation}\label{eq:Ca_eq}\Ca_{\nu_0}\left(\left(D(\h^{reg})\otimes D^\lambda(\mathbb{P}^{r-1})^{\otimes n}\right)^{\mathfrak{S}_n}\right)=
\left(\Ca_{\nu_0}\left(D(\h^{reg})\otimes D^\lambda(\mathbb{P}^{r-1})^{\otimes n}\right)\right)^{\mathfrak{S}_n}\end{equation}
There is a natural homomorphism from the left hand side to the right hand side.
To prove that it is an isomorphism one can argue as follows. First, note, that since the $\mathfrak{S}_n$-action
on $\h^{reg}$ is free, we have
$$D(\h^{reg})\otimes D^\lambda(\mathbb{P}^{r-1})^{\otimes n}=D(\h^{reg})\otimes_{D(\h^{reg})^{\mathfrak{S}_n}}\left(D(\h^{reg})\otimes D^\lambda(\mathbb{P}^{r-1})^{\otimes n}\right)^{\mathfrak{S}_n}$$
Since $D(\h^{reg})$ is ${\nu_0}(\C^\times)$-invariant, the previous equality implies (\ref{eq:Ca_eq}).
%The claim that it is an isomorphism
%follows from two observations: that the $\mathfrak{S}_n$ on $\h^{reg}$ is free and that for any algebra $\A$
%with a $\C^\times$-action and any algebra homomorphism $\mathcal{B}\rightarrow \A^{\C^\times}$ we have
%$(\mathcal{B}\otimes_{\A^{\C^\times}}\A)^0=\mathcal{B}\otimes_{\mathcal{A}^{\C^\times}}\A^0$.

{\it Step 5}. Now let us describe $\Ca_{\nu_0}(\left(D(\h^{reg})\otimes D^\lambda(\mathbb{P}^{r-1})^{\otimes n}\right)=D(\h^{reg})\otimes \Ca_{\nu_0}\left((D^\lambda(\mathbb{P}^{r-1}))^{\otimes n}\right)$. The $\C^\times$-action on the tensor product $(D^\lambda(\mathbb{P}^{r-1}))^{\otimes n}$ is diagonal and it is easy to see that
$\Ca_{\nu_0}\left((D^\lambda(\mathbb{P}^{r-1}))^{\otimes n}\right)=\left(\Ca_{\nu_0}(D^\lambda(\mathbb{P}^{r-1}))\right)^{\otimes n}$.
  So we need to compute $\Ca_{\nu_0}(D^\lambda(\mathbb{P}^{r-1}))$. We claim that this algebra is isomorphic to $\C^{\oplus r}$. Indeed, $D^\lambda(\mathbb{P}^{r-1})$ is a quotient of the central reduction $U_{\tilde{\lambda}}(\mathfrak{sl}_r)$ of $U(\mathfrak{sl}_r)$ at
the central character $\tilde{\lambda}:=\lambda\omega_{r}$. We remark that $\lambda\omega_r+\rho$ is regular because $\lambda\geqslant 0$. We have $\Ca_{\nu_0}(U_{\tilde{\lambda}}(\mathfrak{sl}_r))=\C^{\oplus r!}$ and $\Ca_{\nu_0}(D^\lambda(\mathbb{P}^{r-1}))$ is a quotient of that. The number of irreducible representations of $\Ca_{\nu_0}(D^\lambda(\mathbb{P}^{r-1}))$ equals to the number of simples in the category $\mathcal{O}$ for $D^\lambda(\mathbb{P}^{r-1})$ that coincides with $r$ since abelian localization holds. An isomorphism  $\Ca_{\nu_0}(D^\lambda(\mathbb{P}^{r-1}))=\C^{\oplus r}$ follows.

{\it Step 6}. So we see that $\Ca_{\nu_0}(\bar{\A}_\lambda(n,r)[\delta^{-1}])= \left(D(\h^{reg})\otimes (\C^{\oplus r})^{\otimes n}\right)^{\mathfrak{S}_n}$.
By similar reasons, we have $\Gamma([\bar{\M}^\theta(n,r)_\delta]^{{\nu_0}(\C^\times)}, \Ca_{\nu_0}(\bar{\A}^\theta_\lambda(n,r)))=
\left(D(\h^{reg})\otimes (\C^{\oplus r})^{\otimes n}\right)^{\mathfrak{S}_n}$. The natural homomorphism
\begin{equation}\label{eq:local_iso1}\Ca_{\nu_0}(\bar{\A}_\lambda(n,r)[\delta^{-1}])\rightarrow \Gamma((\bar{\M}^\theta(n,r)_\delta)^{{\nu_0}(\C^\times)}, \Ca_{\nu_0}(\bar{\A}^\theta_\lambda(n,r)))\end{equation}
is an isomorphism by the previous two steps. Also we have a natural homomorphism \begin{equation}\label{eq:local_iso2}\Gamma(\Ca_{\nu_0}(\bar{\A}^\theta_\lambda(n,r)))[\delta^{-1}]\rightarrow
\Gamma\left([\M^\theta(n,r)_\delta]^{{\nu_0}(\C^\times)}, \Ca_{\nu_0}(\bar{\A}^\theta_\lambda(n,r))\right).\end{equation}
The latter homomorphism is an isomorphism from the explicit description
of $\Ca_{\nu_0}(\bar{\A}^\theta_\lambda(n,r))$. Indeed, $\Ca_{\nu_0}(\bar{\A}^\theta_\lambda(n,r))$
is the direct sum of quantizations of products of Hilbert schemes. The morphism $\prod \operatorname{Hilb}_{n_i}(\C^2)
\rightarrow \prod \left(\C^{2n_i}/\mathfrak{S}_n\right)$ is an isomorphism over the non-vanishing locus of $\delta$.
This implies that (\ref{eq:local_iso2}) is an isomorphism.

By the construction,  (\ref{eq:local_iso1}) is the composition of
$\Ca_{\nu_0}(\bar{\A}_\lambda(n,r)[\delta^{-1}])\rightarrow
\Gamma(\Ca_{\nu_0}(\bar{\A}^\theta_\lambda(n,r)))[\delta^{-1}]$ and (\ref{eq:local_iso2}).
So we have proved that  $\Ca_{\nu_0}(\bar{\A}_\lambda(n,r))[\delta^{-1}]\rightarrow \Gamma(\Ca_{\nu_0}(\bar{\A}^\theta_\lambda(n,r)))[\delta^{-1}]$ is
an isomorphism.

{\it Step 7}. For $p\in \bar{\M}^\theta(n,r)^{T}$, let $L^0(p)$ be the corresponding irreducible
$\Gamma(\Ca_{\nu_0}(\bar{\A}^\theta_\lambda(n,r)))$-module from category $\mathcal{O}$. These modules are either finite dimensional
(those are parameterized by the multi-partitions with one part equal to $(n)$ and others empty) or has support of maximal
dimension. It follows from Proposition \ref{Prop:surject} that all finite dimensional $L^0(p)$ restrict
to pairwise non-isomorphic $\Ca_{\nu_0}(\bar{\A}_\lambda(n,r))$-modules. Now consider $L^0(p)$ with support of maximal dimension.
We claim that the localizations $L^0(p)[\delta^{-1}]$ are pairwise non-isomorphic simple $\Gamma(\Ca_{\nu_0}(\bar{\A}^\theta_\lambda(n,r)))[\delta^{-1}]$-modules.
Let us consider $p=(p^1,\ldots,p^r)$ and $p'=(p'^1,\ldots,p'^r)$ with $|p^i|=|p'^i|$ for all $i$ and show that
\begin{itemize}
\item
the  localizations of $L^0(p),L^0(p')$ are simple
\item and, moreover, are isomorphic only if $p=p'$.
\end{itemize}
The analogous claims hold if we localize to the regular locus for $\prod_{i=1}^r\mathfrak{S}_{|p^i|}$.
Indeed, this localization realizes the KZ functor that is a quotient onto its image. So the images of
$L^0(p), L^0(p')$ under this localization are simple and non-isomorphic. The modules $L^0(p)[\delta^{-1}], L^0(p')[\delta^{-1}]$ further restrict   to the locus where $x_i\neq x_j$ for all $i,j$.
But there is no monodromy of the D-modules $L^0(p)[\delta^{-1}],L^0(p')[\delta^{-1}]$
along those additional hyperplanes and these D-modules  have regular singularities everywhere.
It follows that they remain simple and nonisomorphic (if $p\neq p'$).

{\it Step 8}. So we see that the $\Ca_{\nu_0}(\bar{\A}_\lambda(n,r))[\delta^{-1}]$-modules  $L^0(p)[\delta^{-1}]$ are simple and pair-wise non-isomorphic. The $\Ca_{\nu_0}(\bar{\A}_\lambda(n,r))$-module $L^0(p)$ is  finitely generated,
because the algebra $\Gamma(\Ca_{\nu_0}(\bar{\A}^\theta_\lambda(n,r)))$ is a finitely generated $\Ca_{\nu_0}(\bar{\A}_\lambda(n,r))$-module, the latter follows from \cite[Lemma 5.4]{CWR}.
So the $\Ca_{\nu_0}(\bar{\A}_\lambda(n,r))$-module $L^0(p)$ lies in  $\mathcal{O}_\nu(\Ca_{\nu_0}(\bar{\A}_\lambda(n,r)))$
thanks to the weight decomposition.  There is a simple constituent $\underline{L}^0(p)$
of the $\Ca_{\nu_0}(\bar{\A}_\lambda(n,r))$-module $L^0(p)$ with $\underline{L}^0(p)[\delta^{-1}]=L^0(p)[\delta^{-1}]$ because the right hand side
is simple. The finite dimensional modules $L^0(p)$ together with the modules of the form $\underline{L}^0(p)$
give a required number of pairwise nonisomorphic simple $\Ca_{\nu_0}(\bar{\A}_\lambda(n,r))$-modules.
\end{proof}

\subsection{Completion of proofs}\label{SS_loc_compl}
The following proposition completes the proof of Theorem \ref{Thm:loc}.

\begin{Prop}\label{Prop:loc_compl}
Let $\lambda$ be a positive parameter with denominator $n$. Then  abelian localization holds for
$(\lambda,\det)$.
\end{Prop}
\begin{proof}
The proof is in several steps.
%Let $\alpha$ be the one-parameter subgroup $t\mapsto (t^{d_1},\ldots, t^{d_r})$ with $d_1\gg\ldots\gg d_r$.
%Let $\beta: \C^\times\rightarrow T\times \C^\times$ have the form $t\mapsto (1,t)$. Set $\alpha'=m\alpha+\beta$
%for $m\gg 0$. So we have $\alpha\prec^\lambda \alpha'$ for all $\lambda$ thanks to Lemma \ref{Lem:prec_spec}.

{\it Step 1}.
 Since $\Gamma_\lambda^\theta: \mathcal{O}_{\nu}(\bar{\A}^\theta_\lambda(n,r))
\rightarrow \mathcal{O}_{\nu}(\bar{\A}_\lambda(n,r))$ is a quotient functor, to prove that it is an equivalence
it is enough to verify that the number of simples in these two categories is the same. The number of
simples in $\OCat_{\nu}(\bar{\A}_\lambda(n,r))$ coincides with that for $\OCat(\Ca_{\nu_0}(\bar{\A}_\nu(n,r)))$
by the paragraph before Proposition \ref{Prop:simple_numbers}.
The latter is bigger than or equal to the number of simples for $\OCat(\bigoplus \bar{\A}_\lambda(n_1,\ldots,n_r;r))$ that, in its
turn coincides with the number of the $r$-multipartitions of $n$ because  abelian localization holds
for all summands $\bar{\A}_\lambda(n_1,\ldots,n_r;r)$. We deduce that the number of simples
in $\mathcal{O}_{\nu}(\bar{\A}^\theta_\lambda(n,r))$ and in $\mathcal{O}_{\nu}(\bar{\A}_\lambda(n,r))$ coincide.
So we see that $\Gamma^\theta_\lambda:\OCat_{\nu}(\bar{\A}^\theta_\lambda(n,r))\twoheadrightarrow \OCat_{\nu}(\bar{\A}_\lambda(n,r))$
is an equivalence. Now we are going to show that this implies that $\Gamma^\theta_\lambda:\bar{\A}_\lambda^\theta(n,r)\operatorname{-mod}
\rightarrow \bar{\A}_\lambda(n,r)\operatorname{-mod}$ is an equivalence. Below we write $\OCat$
instead of $\OCat_{\nu}$. Our argument is similar to the proof of \cite[Theorem 2.1]{cher_ab_loc}.

{\it Step 2}.
Since $\Gamma_\lambda^\theta $ is an equivalence between the categories $\OCat$, we see that $\bar{\A}^{0}_{\lambda,\chi}(n,r)\otimes_{\bar{\A}_\lambda(n,r)}\bullet$
and $\bar{\A}^{0}_{\lambda+\chi,-\chi}(n,r)\otimes_{\bar{\A}_{\lambda+\chi}(n,r)}\bullet$ are mutually
inverse equivalences between $\OCat(\bar{\A}_\lambda(n,r))$ and $\OCat(\bar{\A}_{\lambda+\chi}(n,r))$ for all $\lambda>0$ and  $\chi\in \Z$ such that $(\lambda+\chi,\det)$ satisfies abelian localization (which happens as long as
$\chi$ is sufficiently large).
Set $\mathcal{B}:=\bar{\A}^{0}_{\lambda+\chi,-\chi}(n,r)
\otimes_{\bar{\A}_{\lambda+\chi}(n,r)}
\bar{\A}^{0}_{\lambda,\chi}(n,r)$.
This is a HC $\bar{\A}_\lambda(n,r)$-bimodule with a natural homomorphism to $\bar{\A}_\lambda(n,r)$ such that
the induced homomorphism $\mathcal{B}\otimes_{\bar{\A}_\lambda(n,r)}M\rightarrow M$ is an isomorphism
for any $M\in \OCat(\bar{\A}_\lambda(n,r))$. On the other hand, for any nonzero $x\in \bar{\M}(n,r)$,
the bimodules $\bar{\A}^{0}_{\lambda+\chi,-\chi}(n,r), \bar{\A}^{0}_{\lambda,\chi}(n,r)$ are mutually
inverse Morita equivalences, this follows from Lemma \ref{Lem:transl_properties} since abelian
localization holds for all slice algebras with parameters whose denominator is $n$.
It follows that both kernel and cokernel of $\mathcal{B}\rightarrow \bar{\A}_\lambda(n,r)$
are finite dimensional.

{\it Step 3}.
Let $L$ denote an irreducible finite dimensional $\bar{\A}_\lambda(n,r)$-module, it is unique
because of the equivalence $\OCat(\bar{\A}_\lambda(n,r))\cong \OCat(\bar{\A}_{\lambda+\chi}(n,r))$
and Proposition \ref{Prop:support_local}. Since the homomorphism
$\mathcal{B}\otimes_{\bar{\A}_\lambda(n,r)}L\rightarrow L$ is an isomorphism,  we see that $\mathcal{B}\twoheadrightarrow \bar{\A}_\lambda(n,r)$. Let $K$ denote the kernel.  We have an exact sequence
$$\operatorname{Tor}^1_{\bar{\A}_\lambda(n,r)}(\bar{\A}_\lambda(n,r),L)\rightarrow K\otimes_{\bar{\A}_\lambda(n,r)}L
\rightarrow \mathcal{B}\otimes_{\bar{\A}_\lambda(n,r)}L\rightarrow L\rightarrow 0$$
Clearly, the first term is zero, while the last homomorphism is an isomorphism. We deduce that $K\otimes_{\bar{\A}_\lambda(n,r)}L=0$.
But $K$ is a finite dimensional $\bar{\A}_\lambda(n,r)$-bimodule and hence a $\bar{\A}_\lambda(n,r)/\operatorname{Ann}L$-bimodule. So its tensor product with $L$ can only be zero if $K=0$.

{\it Step 4}.
So we see that $\bar{\A}^{0}_{\lambda+\chi,-\chi}(n,r)\otimes_{\bar{\A}_{\lambda+\chi}(n,r)}\bar{\A}^{0}_{\lambda,\chi}(n,r)\cong
\bar{\A}_{\lambda}(n,r)$. Similarly, $\bar{\A}^{0}_{\lambda,\chi}(n,r)\otimes_{\bar{\A}_{\lambda}(n,r)}\bar{\A}^{0}_{\lambda+\chi,-\chi}(n,r)\cong
\bar{\A}_{\lambda+\chi}(n,r)$. It follows that $\Gamma_\lambda^\theta$ is an equivalence between $\bar{\A}_\lambda^\theta(n,r)\operatorname{-mod}$ and $ \bar{\A}_\lambda(n,r)\operatorname{-mod}$.
\end{proof}

\begin{proof}[Proof of Theorem \ref{Thm:fin dim}]
%Proposition \ref{Prop:support_local} together with Theorem \ref{Thm:loc} (and the claim that $\Gamma_\lambda^{\det}$
%is a quotient functor for $\lambda>-r$) imply that
%\begin{itemize}
%\item
%\end{itemize}
%It remains to show that $\bar{\A}_\lambda(n,r)$
%with $-r<\lambda<0$ has no finite dimensional irreducible representations. Assume the converse,
Let $L$
denote a finite dimensional irreducible representation of $\bar{\A}_\lambda(n,r)$. Since $L\Loc_\lambda^\theta(\bar{\A}_\lambda(n,r))=\bar{\A}_\lambda^\theta(n,r)$
and $R\Gamma_\lambda^\theta(\bar{\A}^\theta_\lambda(n,r))=\bar{\A}_\lambda(n,r)$, we see that
$R\Gamma_\lambda^\theta\circ L\Loc_\lambda^\theta$ is the identity functor of $D^-(\bar{\A}_\lambda(n,r)\operatorname{-mod})$. The homology of $L\Loc_\lambda^\theta(L)$ are supported
on $\bar{\rho}^{-1}(0)$. From Proposition \ref{Prop:support_local} it follows that the denominator of $\lambda$ is $n$.
Thanks to that proposition combined with Theorem \ref{Thm:loc}, the present proof reduces to showing
that for $\lambda$ between $-r$ and $0$ and  with denominator $n$, the algebra $\bar{\A}_\lambda(n,r)$
has no finite dimensional representations.

Recall that $\Gamma_\lambda^\theta$ is an exact functor. Since $R\Gamma_\lambda^\theta\circ L\Loc_\lambda^\theta$
is the identity, the functor $\Gamma_\lambda^\theta$ does not kill
the simple $\bar{\A}_\lambda^\theta(n,r)$-module $\tilde{L}$ supported on $\bar{\rho}^{-1}(0)$.
On the other hand, $\Gamma_\lambda^\theta$ does not kill modules whose support intersects $\bar{\M}^\theta(n,r)^{reg}$,
the open subvariety in $\bar{\M}^\theta(n,r)$, where $\bar{\rho}$ is an isomorphism. In fact, every simple
in $\OCat(\bar{\A}_\lambda^\theta(n,r))$ is either supported on $\bar{\rho}^{-1}(0)$ (if it is homologically shifted
under the wall-crossing functor so that the global sections are finite dimensional) or its support intersects
$\bar{\M}^\theta(n,r)^{reg}$ (if it is not).

So we see that $\Gamma_\lambda^\theta$ does not kill any irreducible module in $\OCat(\bar{\A}^\theta_\lambda(n,r))$.
So it is an equivalence. By the proof of Proposition \ref{Prop:loc_compl}, $(\lambda,\theta)$ satisfies abelian
localization, which is impossible as we already know.
\end{proof}

\section{Two-sided ideals and dimensions of supports}
\subsection{Two-sided ideals}
The goal of this section is to prove Theorem \ref{Thm:ideals}.
We use the following notation. We write $\A$ for $\bar{\A}_\lambda(n,r)$ (where $\lambda$ is not of the form
$\frac{s}{n'}$ with $n'\leqslant n$ coprime to $s$ and $-rn'<s<0$) and write $\underline{\A}$
for $\bar{\A}_\lambda(n',r)$.

Let us start with the description of the two-sided ideals in $\underline{\A}$.

\begin{Lem}\label{Lem:ideals_easy}
There is a unique
proper ideal in $\underline{\A}$.
\end{Lem}
\begin{proof}
The proper slice algebras for $\underline{\A}$ have no finite dimensional representations, compare
to the proof of Proposition \ref{Prop:support_local}. So every ideal $\J\subset \underline{\A}$ is either of finite codimension or $\VA(\underline{\A}/\J)=\bar{\M}(n',r)$.
The algebra $\underline{\A}$ has no zero divisors so the second option is only possible when $\J=\{0\}$. Now suppose that $\J$
is of finite codimension. Then $\underline{\A}/\J$ (viewed as a left $\underline{\A}$-module) is the sum of several copies of the finite dimensional irreducible $\underline{\A}$-module. So $\J$ coincides with the annihilator of the finite dimensional irreducible module, and we are done.
\end{proof}

Let $\underline{\J}$ denote the unique two-sided ideal.

Now we are going to describe the two-sided ideals in $\underline{\A}^{\otimes k}$. For this we need some notation. Set
$\underline{\I}_i:=\underline{\A}^{\otimes i-1}\otimes \underline{\J}\otimes \underline{\A}^{\otimes k-i-1}$. For a subset
$\Lambda\subset \{1,\ldots,k\}$ define the ideals $\underline{\I}_{\Lambda}:=\sum_{i\in \Lambda} \underline{\I}_i,
\underline{\I}^{\Lambda}:=\prod_{i\in \Lambda} \underline{\I}_i$.

Recall that  a collection of subsets in $\{1,\ldots,k\}$ is called an {\it anti-chain} if none of these subsets is contained
in another. Also recall that an ideal $I$ in an associative algebra $A$ is called {\it semi-prime} if it is the intersection of
prime ideals.

\begin{Lem}\label{Lem:ideals_next}
The following is true.
\begin{enumerate}
\item The prime ideals in $\underline{\A}^{\otimes k}$ are precisely the ideals $\underline{\I}_\Lambda$.
\item For every ideal $\I\subset \underline{\A}^{\otimes k}$, there is a unique anti-chain $\Lambda_1,\ldots,\Lambda_q$
of subsets in $\{1,\ldots,k\}$ such that $\I=\bigcap_{i=1}^p \I_{\Lambda_i}$. In particular, every ideal is
semi-prime.
\item For every ideal $\I\subset \underline{\A}^{\otimes k}$, there is a unique anti-chain $\Lambda_1',\ldots,\Lambda_q'$
of subsets of $\{1,\ldots,k\}$ such that $\I=\sum_{i=1}^q \I^{\Lambda'_i}$.
\item The anti-chains in (2) and (3) are related as follows: from an antichain in (2), we form all possible subsets
containing an element from each of $\Lambda_1,\ldots,\Lambda_p$. The minimal such subsets form an anti-chain in (3).
\end{enumerate}
\end{Lem}
The proof essentially appeared in \cite[5.8]{sraco}.
\begin{proof}
Let us prove (1). Let $\I$ be a prime ideal. Let $x$ be a generic  point in an open leaf  $\Leaf\subset\VA(\underline{\A}^{\otimes k}/\I)$
of maximal dimension. The corresponding slice algebra $\underline{\A}'$ has a finite dimensional representation and so is again the product of several copies of $\underline{\A}$. The leaf $\Leaf$ is therefore the product of one-point leaves and open leaves in
$\bar{\M}(n',r)^k$.  An irreducible finite dimensional representation of $\underline{\A}'$ is unique, let $\I'$ be its annihilator.

Consider the categories $\HC_{\overline{\Leaf}}(\underline{\A}^{\otimes k})$ of all HC $\underline{\A}^{\otimes k}$-bimodules whose associated variety is contained in $\overline{\Leaf}$ and $\HC_{fin}(\A')$ of finite dimensional $\underline{\A}'$-bimodules (that are automatically HC). So the functor $\bullet_{\dagger,x}$ restricts to
$\HC_{\overline{\Leaf}}(\underline{\A}^{\otimes k})\rightarrow \HC_{fin}(\underline{\A}')$. As we have mentioned in
Section \ref{SS_HC}, this functor admits a right adjoint
$$\bullet^{\dagger,x}: \HC_{fin}(\underline{\A}')\rightarrow
\HC_{\overline{\Leaf}}(\underline{\A}^{\otimes k}).$$

Let $\I^1$ denote the kernel of the natural homomorphism $\underline{\A}^{\otimes k}
\rightarrow (\underline{\A}'/\I')^{\dagger,x}$ and $\I\subset \I^1$.
So $\VA(\underline{\A}^{\otimes k}/\I^1)=\overline{\Leaf}$.
It follows from \cite[Corollar 3.6]{BoKr} that $\I=\I^1$. So the number of the prime ideals coincides
with that of the non-empty subsets  $\{1,\ldots,k\}$. On the other hand, the ideals $\I^\Lambda$ are all different
(they have different associated varieties)
and all prime (the quotient $\underline{\A}^{\otimes k}/\I^{\Lambda}$ is the product of a matrix algebra and the algebra
$\underline{\A}^{\otimes k-|\Lambda|}$ that has no zero divisors).

Let us prove (2) (and simultaneously (3)). Let us write $\I_{\Lambda_1,\ldots,\Lambda_p}$ for $\bigcap_{j=1}^s \I_{\Lambda_j}$. For ideals in $\underline{\A}^{\otimes k-1}$ we use notation like $\underline{\I}_{\Lambda_1',\ldots,\Lambda_q'}$. Reordering the indexes, we may assume that $k\in \Lambda_1,\ldots,\Lambda_s$ and $k\not\in \Lambda_{s+1},\ldots,\Lambda_p$. Set $\Lambda_j':=\Lambda_j\setminus \{k\}$ for $j\leqslant s$. Then
\begin{equation}\label{eq:ideal1}\I_{\Lambda_1,\ldots,\Lambda_p}=(\underline{\A}^{\otimes k-1}\otimes \underline{\J}+ \underline{\I}_{\Lambda'_1,\ldots,\Lambda'_s}\otimes \underline{\A})\cap (\underline{\I}_{\Lambda_{s+1},\ldots,\Lambda_p}\otimes \underline{\A}).\end{equation}
We claim that the right hand side of (\ref{eq:ideal1}) coincides with
\begin{equation}\label{eq:ideal2}
\underline{\I}_{\Lambda_{s+1},\ldots,\Lambda_p}\otimes \underline{\J}+\underline{\I}_{\Lambda_1',\ldots,\Lambda_s',\Lambda_{s+1},\ldots,\Lambda_p}\otimes \underline{\A}.
\end{equation}
First of all, we notice that (\ref{eq:ideal2}) is contained in (\ref{eq:ideal1}). So we only need to prove the
opposite inclusion.   The projection  of (\ref{eq:ideal1}) to $\underline{\A}^{\otimes k-1}\otimes
(\underline{\A}/\underline{\J})$ is contained in $\underline{\I}_{\Lambda_1',\ldots,\Lambda_s',\Lambda_{s+1},\ldots,\Lambda_p}$
and hence also in the projection of (\ref{eq:ideal2}). Also the intersection of (\ref{eq:ideal1}) with $\underline{\A}^{\otimes k-1}\otimes \underline{\J}$ is contained in $\underline{\I}_{\Lambda_{s+1},\ldots,\Lambda_p}\otimes \underline{\J}$. So (\ref{eq:ideal1})
is included into (\ref{eq:ideal2}).

Repeating this argument with the sum similar to (\ref{eq:ideal2}) but for other $k-1$ factors of $\underline{\A}^{\otimes k}$ we conclude that $\I_{\Lambda_1,\ldots,\Lambda_p}=\sum_j \I^{\Lambda'_j}$, where the subsets
$\Lambda'_j\subset \{1,\ldots,k\}$ are formed as described in (4).   So we see that the ideals (2) are the same
as the ideals in (3) and that (4) holds. What remains to do is to prove that every ideal has the form described in (2).
To start with, we notice that every semi-prime ideal has the form as in (2) because of (1). In particular, the radical
of any ideal has such form.

Clearly,
$\I^{\Lambda'_1}\I^{\Lambda'_2}=\I^{\Lambda'_1\cup\Lambda'_2}$. So it follows that any sum of the ideals $\I^{\Lambda_j'}$ coincides with its
square. So if $\I$ is an ideal whose radical is $\I_{\Lambda_1,\ldots,\Lambda_p}$, then $\I$ coincides with
its radical. This completes the proof.
\end{proof}

Now we are ready to establish a result that will imply Theorem \ref{Thm:ideals} for nonsingular parameters $\lambda$. Let $x_i\in \bar{\M}(n,r)$ be a point corresponding to the leaf with slice $\bar{\M}(n',r)^{i}$ (i.e. to the semisimple representations of the form $r^0\oplus (r^1)^{n'}\oplus\ldots\oplus (r^i)^{n'}$). We set
$\J_i=\ker[\A\rightarrow \left(\underline{\A}/\underline{\I})^{\otimes i}\right)^{\dagger,x_i}]$.

\begin{Prop}\label{Prop:ideals_techn}
The ideals $\J_i, i=1,\ldots,q$, have the following properties.
\begin{enumerate}
\item The ideal $\J_i$ is prime for any $i$.
\item $\VA(\A/\J_i)=\overline{\Leaf}_i$, where $\Leaf_i$ is the symplectic leaf containing $x_i$.
\item $\J_1\subsetneq \J_2\subsetneq\ldots\subsetneq \J_q$.
\item Any proper two-sided ideal in $\A$ is one of $\J_i$.
%\item We have $(\J_i)_{\dagger,x_j}=\underline{\A}^{\otimes j}$ if $j<i$ and $(\J_i)_{\dagger,x_j}=\sum_{|\Lambda|=j-i+1} %\I^{\Lambda}$ else.
\end{enumerate}
\end{Prop}
\begin{proof}
(2) follows from the construction. Also from the construction it follows that $\J_i$ is the maximal
among ideal with given associated variety. So it is prime similarly to the proof of (1) of Lemma
\ref{Lem:ideals_next}.

Let us prove (3). Since $(\J_{i})_{\dagger,x_i}$ has finite codimension, we see that it coincides with the maximal ideal
in $\underline{\A}^{\otimes i}$. So $(\J_j)_{\dagger,x_i}\subset (\J_i)_{\dagger,x_i}$ for $j<i$. Again from
the construction of $\J_i$, it follows that
$\J_j\subsetneq\J_i$.

Let us prove (4). The functor $\bullet_{\dagger,x_q}$ is faithful. Indeed, otherwise we have a HC bimodule
$\M$ with $\VA(\M)\cap \Leaf_q=\varnothing$. But $\M_{\dagger,x}$ has to be nonzero finite dimensional
for some $x$ and this is only possible when $x\in \Leaf_i$ for some $i$. But $\Leaf_q\subset \overline{\Leaf}_i$
for all $i$ that shows faithfulness. Since $\bullet_{\dagger,x_q}$ is faithful and exact, it follows that it embeds the lattice of
the ideals in $\A$ into that in $\underline{\A}^{\otimes q}$. We claim that this implies that every ideal in
$\A$ is semiprime. Indeed, the functor $\bullet_{\dagger,x_q}$ is, in addition, tensor and so
preserves products of ideals. In particular, any two-sided ideal in $\A$ coincides with its square.
Our claim follows from (2) of Lemma \ref{Lem:ideals_next}. But every prime ideal
in $\A$ is some $\J_i$, this is proved analogously to (1) of Lemma \ref{Lem:ideals_next}. Since the ideals
$\J_i$ form a chain, any semiprime ideal is prime and so coincides with some $\J_i$.
%
%Let us prove (5). We will deduce that from the behavior of $\bullet_{\dagger,x}$ on the associated varieties.
%We have an action of $\mathfrak{S}_j$ on $\M(n',r)^j$ by permuting factors. The action is induced
%from $N_{G}(G_{\tilde{x}})$, where $\tilde{x}$ is a point from the closed $G$-orbit lying over $x$.
%It follows that the intersection of any leaf with the slice is $\mathfrak{S}_j$-stable. The associated
%variety $\VA(\underline{\A}^{\otimes j}/(\J_i)_{\dagger,x_j})$ is the union of some products with factors
%$\{\operatorname{pt}\}$ and $\bar{\M}(n',r)$, where, for the dimension reasons, $\bar{\M}(n',r)$
%occurs $j-i$ times. Because of the $\mathfrak{S}_j$-symmetry, all products occur. Now we deduce
%the required formula for $(\J_i)_{\dagger,x_j}$ from the description of the two-sided ideals
%in $\bar{\A}^{\otimes j}$. This description shows that for each associated variety there is at most
%one two-sided ideal.
\end{proof}

\begin{proof}
The case of finite homological dimension follows from Proposition \ref{Prop:ideals_techn}.
Let us now consider the case when $\lambda=\frac{s}{n'}$ with $-rn'<s<0$. The algebra
$\bar{\A}_\lambda(n,r)$ has no finite dimensional representations and neither does
any of the slice algebras $\bigotimes_{i=1}^k \bar{\A}_\lambda(n_i,r)$. By using the
restriction functors (say similarly to Step 3 of the proof of Proposition
\ref{Prop:support_local}) we see that the algebra $\bar{\A}_\lambda(n,r)$ is simple.
\end{proof}

\subsection{Restriction functors for asymptotic chamber}
In this section we assume that $\nu$ is dominant meaning that
$\nu(t)=(\operatorname{diag}(t^{d_1},\ldots, t^{d_r}),t)$ with
$d_1\gg d_2\gg\ldots\gg d_r$. We also assume that $\lambda>0$ and is Zariski
generic.

Let $\tau=(\tau_1,\ldots,\tau_k)$ be a partition of $n$.
Set $\OCat_{\nu,\lambda}(\tau,r):=\boxtimes_{i=1}^k \OCat_\nu(\A_\lambda(\tau_i,r))$.
We will produce exact functors $\Res_\tau:\OCat_{\nu,\lambda}(n,r)\rightarrow
\OCat_{\nu,\lambda}(\tau,r)$ generalizing the Bezrukavnikov-Etingof functors, \cite{BE}, for
$r=1$.

The filtration by the order of differential operators on $D(R)$ induces a filtration
on $\A_\lambda(n,r)$. The degree zero part is $\C[R]^G=\C[\bar{\g}]^G$. Pick a point
$b\in \bar{\g}\quo G$ such that the stabilizer of the corresponding closed orbit is
$\prod_{i=1}^k \GL(\tau_i)$ and consider the tensor $\A_\lambda(n,r)^{\wedge_b}:=
\C[\bar{\g}\quo G]^{\wedge_b}\otimes_{\C[\bar{\g}]^G}\A_\lambda(n,r)$. Since the adjoint action of $\C[\bar{\g}]^G$
is locally nilpotent, $\A_\lambda(n,r)^{\wedge_b}$ is naturally an algebra.
Set $\A_\lambda(\tau,r)^{\wedge_0}:=\boxtimes_{i=1}^k \A_\lambda(\tau_i,r)^{\wedge_0}$.

\begin{Lem}\label{Lem:iso_compl}
We have a $\GL(r)$-equivariant isomorphism of filtered algebras $\vartheta:\A_\lambda(n,r)^{\wedge_b}\xrightarrow{\sim}
\boxtimes_{i=1}^k \A_\lambda(\tau_i,r)^{\wedge_0}$ (the action on the right hand side is diagonal).
\end{Lem}
\begin{proof}
Consider the Rees algebras $\A_{\lambda \hbar}(n,r), \A_{\lambda\hbar}(\tau,r)$
and their full completions $\A_{\lambda\hbar}(n,r)^{\wedge_b}$ (at the point with closed
$G$-orbit  given by a diagonal matrix in $\g\subset \mu^{-1}(0)$ corresponding to $b$)
and $\A_{\lambda\hbar}(\tau,r)^{\wedge_0}$. As was mentioned in Section \ref{SS_sympl_slices},
see (\ref{eq:quant_compl_iso}),
we have a $\C[\hbar]$-linear
isomorphism $\A_{\lambda\hbar}(n,r)^{\wedge_b}\cong \A_{\lambda\hbar}(\tau,r)^{\wedge_0}$
that can be made $\GL(r)\times \C^\times$-equivariant (here $\C^\times$ is the
contracting action) because we complete at the $\GL(r)\times \C^\times$-stable points.
By taking $\C^\times$-finite parts and taking the quotients by $\hbar-1$ we get
an isomorphism in the lemma.
\end{proof}

We will use an isomorphism from Lemma \ref{Lem:iso_compl} to produce a functor
$\Res_\tau$. First we need to establish an equivalence of the category
$\OCat_{\nu,\lambda}(\tau,r)$ with a certain category of $\A_\lambda(\tau,r)^{\wedge_0}$-modules.
As usual, set $\nu_0(t)=(\operatorname{diag}(t^{d_1},\ldots, t^{d_r}),1)$.
We consider the category $\OCat_{\nu,\lambda}(\tau,r)^{\wedge_0}$ consisting of all
finitely generated $\A_\lambda(\tau,r)^{\wedge_0}$-modules such that $h_0=d_1\nu_0$
 acts locally finitely with eigenvalues bounded from above and generalized eigenspaces that are finitely generated
over $\C[\g_\tau\quo G_\tau]^{\wedge_0}$, where $\g_\tau$ is the standard Levi subalgebra of $\g$
corresponding to $\tau$. Note that all generalized $h_0$-eigenspaces in a module
from $\OCat_{\nu,\lambda}(n,r)$ are finitely generated over $\C[\g_\tau]^{G_\tau}$.
So we get an exact functor
$$N\mapsto N^{\wedge_0}:=\C[\g_\tau\quo G_\tau]^{\wedge_0}\otimes_{\C[\g_\tau]^{G_\tau}}N:\OCat_{\nu^+,\lambda}(\tau,r)\rightarrow \OCat_{\nu,\lambda}(\tau,r)^{\wedge_0}.$$

\begin{Lem}\label{Lem:compl_equi}
The functor $\bullet^{\wedge_0}$ is a category equivalence. A quasi-inverse functor is given by
taking the $h$-finite elements, where $h$ is the image of $1$ under the quantum comoment map
for $t\mapsto \nu(t)\nu_0(t)^{-1}$.
\end{Lem}
\begin{proof}
Let $N'_{fin}$ stand for the space of $h$-finite elements. It is easy to see that $N'_{fin}$
is the sum of modules from $\OCat_{\nu,\lambda}(\tau,r)$. Note that all simultaneous generalized eigenspaces
for $(h,h_0)$ are finite dimensional. This is because the generalized eigenspaces for $h_0$ are finitely
generated modules over $\C[\g_{\tau}\quo G_\tau]^{\wedge_0}$ and the generalized eigenspaces for $h$
in such modules are finite dimensional.
So $N'_{fin}$ actually lies in $\OCat$. Also $N'_{fin}$
is dense in $N'$. Now the proof is easy.
\end{proof}

We note that, for $M\in \OCat_\nu(\A_\lambda(n,r))$, the $\A_\lambda(\tau,r)^{\wedge_0}$-module
$\vartheta_*\left( \C[\g\quo G]^{\wedge_b}\otimes_{\C[\g]^G}M\right)$ lies in
$\OCat_{\nu,\lambda}(\tau,r)^{\wedge_0}$.
We define $\Res_{\tau}(M)$ by $[\vartheta_*\left( \C[\g\quo G]^{\wedge_b}\otimes_{\C[\g]^G}M\right)]_{fin}$.
This is an exact functor by construction.

Now let us study  properties of $\Res_\tau$.

First of all, the restriction functor behaves nicely on the level of associated varieties.
The following result is a straightforward consequence of the construction.

\begin{Lem}\label{Lem:assoc_var_restr}
The associated variety $\VA(\Res_\tau(M))$ is a unique conical (with respect to the contracting
$\C^\times$-action) subvariety in $\M(\tau,r)$ such that $\VA(\Res_\tau(M))\cap \M(\tau,r)^{\wedge_0}=
\VA(M)\cap \M(n,r)^{\wedge_b}$ (where we consider the full completions).
\end{Lem}

The following basic property is extremely important. We can consider the functor
$$\underline{\Res}_\tau:\OCat_{\nu}(\Ca_{\nu_0}(\A_\lambda(n,r)))
\rightarrow \OCat_{\nu}(\Ca_{\nu_0}(\A_\lambda(\tau,r))).$$
It is defined in the same way as $\Res_\tau$.
On the summand corresponding to a composition $\mu=(n_1,\ldots,n_r)$,
the functor $\underline{\Res}_{\tau}$ coincides with the
direct sum of suitable Bezrukavnikov-Etingof restriction functors. More precisely, the corresponding summand of
$\Ca_{\nu_0}(\A_\lambda(n,r))$ is $\A_\lambda(n_1,1)\otimes \A_{\lambda+1}(n_2,1)\otimes\ldots\otimes
\A_{\lambda+r-1}(n_r,1)$. Then
\begin{equation}\label{eq:underline_R_comput}
\underline{\Res}_\tau(M)=\bigoplus_{S_\mu/W'} \Res^{S_\mu}_{W'} M,
\end{equation}
where the summation is the summation is taken over all $S_\mu$-orbits on $S/S_\tau$, here
$W'$ stands for the standard parabolic stabilizer of the orbit. We write $\Res^{S_\mu}_{W'}$
for the Bezrukavnikov-Etingof restriction functor to the parabolic subgroup $W'$ (or, more precisely,
a version, where we do not change the space $\h$).

We still have the Verma module functor $\Delta_{\nu_0}: \OCat_{\nu}(\Ca_{\nu_0}(\A_\lambda(\tau,r)))
\rightarrow \OCat_{\nu}(\Ca_{\nu_0}(\A_\lambda(n,r)))$.

\begin{Lem}\label{Lem:commut_assoc}
We have  isomorphisms of functors
$$\Delta_{\nu_0}\circ \underline{\Res}_\tau\cong
\Res_\tau\circ \Delta_{\nu_0},
\nabla_{\nu_0}\circ \underline{\Res}_\tau\cong
\Res_\tau\circ \nabla_{\nu_0}$$
\end{Lem}
\begin{proof}
Recall that by the construction, $\Res_\tau$ is isomorphic to the completion functor $\bullet^{\wedge_b}$.

Let us prove the first isomorphism in the lemma. Note first that
\begin{align*}
\A_\lambda(n,r)^{\wedge_b}/\A_\lambda(n,r)^{\wedge_b}\left(\A_\lambda(n,r)^{\wedge_b}\right)^{>0,\nu_0}
\cong \C[\g\quo G]^{\wedge_b}\otimes_{\C[\g]^G}\left(
\A_\lambda(n,r)/\A_\lambda(n,r)\A_\lambda(n,r)^{>0,\nu_0}\right)
\end{align*}
We can consider the functors in the first isomorphism as functors
$\OCat_\nu(\Ca_{\nu_0}(\A_\lambda(n,r)))\rightarrow \A_\lambda(n,r)^{\wedge_b}\operatorname{-mod}$.
Then the right hand side is given by taking the tensor product over $\Ca_{\nu_0}(\A_\lambda(n,r))$
with  $\A_\lambda(n,r)^{\wedge_b}/\A_\lambda(n,r)^{\wedge_b}\left(\A_\lambda(n,r)^{\wedge_b}\right)^{>0,\nu_0}$.
Since the isomorphism $\A_\lambda(n,r)^{\wedge_b}\cong \A_\lambda(\tau,r)^{\wedge_0}$ is $\nu_0(\C^\times)$-equivariant,
we get
$$\A_\lambda(n,r)^{\wedge_b}/\A_\lambda(n,r)^{\wedge_b}\left(\A_\lambda(n,r)^{\wedge_b}\right)^{>0,\nu_0}
\cong \A_\lambda(\tau,r)^{\wedge_0}/\A_\lambda(\tau,r)^{\wedge_0}\left(\A_\lambda(\tau,r)^{\wedge_0}\right)^{>0,\nu_0}.$$
The bimodule in the right hand side coincides with
$$\left(\A_\lambda(\tau,r)/\A_\lambda(\tau,r)\A_\lambda(\tau,r)^{>0,\nu_0}\right)
\otimes_{\C[\g_\tau]^{G_\tau}}\C[\g_\tau\quo
G_\tau]^{\wedge_0}.$$
It follows that the functor in the left hand side of the first isomorphism in the lemma
is given by taking tensor product over $\Ca_{\nu_0}(\A_\lambda(n,r))$ with
$\A_\lambda(\tau,r)^{\wedge_0}/\A_\lambda(\tau,r)^{\wedge_0}\left(\A_\lambda(\tau,r)^{\wedge_0}\right)^{>0,\nu_0}$.
An isomorphism $\Res_\tau\circ \Delta_{\nu_0}\cong \Delta_{\nu_0}\circ \underline{\Res}_\tau$.
%The left hand side is given by the tensor product with %$$(\A_\lambda(n,r)/\A_\lambda(n,r)\A_\lambda(n,r)^{>0,\nu_0})\otimes_{\C[\g]^G}\C[\g\quo G]^{\wedge_b},$$
%while the right hand side is the tensor product with $$\C[\g\quo G]^{\wedge_b}\otimes_{\C[\g]^G}
%(\A_\lambda(n,r)/\A_\lambda(n,r)\A_\lambda(n,r)^{>0,\nu_0}).$$ Both bimodules
%are isomorphic to $\A_\lambda(n,r)^{\wedge_b}/\A_\lambda(n,r)^{\wedge_b}(\A_\lambda(n,r)^{\wedge_b})^{>0,\nu_0}$.
%

Let us proceed to the second isomorphism. Similarly to the first one, both functors are isomorphic to
$$\Hom_{\Ca_{\nu_0}(\A_\lambda(n,r)^{\wedge_b})}
(\A_{\lambda}(n,r)^{\wedge_b}/\left(\A_{\lambda}(n,r)^{\wedge_b}\right)^{<0,\nu_0}\A_\lambda(n,r)^{\wedge_b},
\Ca_{\nu_0}(\A_\lambda(n,r)^{\wedge_b})\otimes_{\Ca_{\nu_0}(\A_\lambda(n,r))}\bullet),$$
where we take the restricted Hom (with respect to the natural grading on the first argument).
\end{proof}

For $M\in \OCat_{\nu}(\Ca_{\nu_0}(\A_\lambda(\tau,r)))$ we write
$L_{\nu_0}(M)$ for the maximal quotient of $\Delta_{\nu_0}(M)$ that does not intersect the highest
weight subspace $M$. Equivalently, $L_{\nu_0}(M)$ is the image of the natural homomorphism
$\Delta_{\nu_0}(M)\rightarrow \nabla_{\nu_0}(M)$ (induced by the identity map $M\rightarrow M$).

The following corollary of Lemma \ref{Lem:commut_assoc} will play a crucial role in computing the
annihilators of simple objects in $\OCat_{\nu,\lambda}(n,r)$.

\begin{Cor}\label{Cor:Res_simple}
For $M\in \OCat_{\nu}(\Ca_{\nu_0}(\A_\lambda(n,r)))$, we have $\Res_\tau(L_{\nu_0}(M))\cong
L_{\nu_0}(\underline{\Res}_\tau M)$.
\end{Cor}
\begin{proof}
We can assume that $M\in \OCat_{\nu}(\Ca_{\nu_0}(\A_\lambda(n,r))|_{\tau'})$
(where $\A_\lambda(n,r))|_{\mu}$ is the summand of $\Ca_{\nu_0}(\A_\lambda(n,r))$
corresponding to a composition $\mu$).
The natural map $\Delta_{\nu_0}(M)\rightarrow \nabla_{\nu_0}(M)$ gives
rise to a map $\Res_\tau\circ \Delta_{\nu_0}(M)\rightarrow \Res_\tau\circ \nabla_{\nu_0}(M)$
that is the identity on the highest weight spaces for $h_0$. The identifications
$$\Res_\tau\circ \Delta_{\nu_0}(M)\cong \Delta_{\nu_0}\circ \underline{\Res}_\tau(M),
\Res_\tau\circ \nabla_{\nu_0}(M)\cong \nabla_{\nu_0}\circ \underline{\Res}_\tau(M)$$
are the identity on the highest weight spaces, by the construction.
So applying the functor $\Res_\tau$ to the morphism $\Delta_{\nu_0}(M)\rightarrow \nabla_{\nu_0}(M)$
we get the natural morphism $\Delta_{\nu_0}\circ \underline{\Res}_\tau(M)\rightarrow
\nabla_{\nu_0}\circ \underline{\Res}_\tau(M)$. Since $\Res_\tau$ is exact, our claim
follows.
\end{proof}

To finish this section let us study the interaction of $\Res_\tau$ with highest weight structures.

\begin{Prop}
The functor $\Res_\tau$ maps (co)standard objects to (co)standardly filtered ones,
tiltings to tiltings, projectives to projectives, injectives to injectives.
\end{Prop}
\begin{proof}
The proof is in several steps.

{\it Step 1}. Let us show that $\Res_\tau$ maps standard objects to standardly filtered ones.
Lemma \ref{Lem:commut_assoc} reduces this statement to showing that $\underline{\Res}_\tau$
maps standard objects to standardly filtered ones. This property of $\underline{\Res}_\tau$  follows from
\cite[Section 2]{Shan}.

{\it Step 2}. Similarly we see that $\Res_\tau$ maps costandard objects into costandardly filtered ones.
Because of this, $\Res_\tau$ sends tiltings to tiltings.

{\it Step 3}. Let us show that $\Res_\tau$ maps injectives to injectives.
The functor $\WC_{\lambda\leftarrow \lambda^-}=
\A^{0}_{\lambda^-,\lambda-\lambda^-}(n,r)\otimes^L_{\A_{\lambda^-}(n,r)}\bullet$
is a Ringel duality functor and hence induces equivalences
$\OCat_\nu(\A_{\lambda^-}(n,r))\operatorname{-tilt}\xrightarrow{\sim} \OCat_\nu(\A_{\lambda}(n,r))\operatorname{-inj}$.
It follows that  $\A^{0}_{\lambda^-,\lambda-\lambda^-}(n,r)\otimes_{\A_{\lambda}^-(n,r)}\bullet$
gives this equivalence. On the other hand from the definition of $\Res_\tau$
we see that \begin{equation}\label{eq:Res_wc}\Res_\tau(\A^{0}_{\lambda^-,\lambda-\lambda^-}(n,r)\otimes_{\A_{\lambda^-}(n,r)}\bullet)\cong
\A^{0}_{\lambda^-,\lambda-\lambda^-}(n,r)_{\dagger,\tau}\otimes_{\A_{\lambda^-}(\tau;r)}\Res_\tau(\bullet).\end{equation}
But the bimodule $\A^{0}_{\lambda^-,\lambda-\lambda^-}(n,r)_{\dagger,\tau}$ is a wall-crossing bimodule
for the algebras $\A_?(\tau;r)$. So taking tensor product with this bimodule maps tiltings to injectives.
So the right hand side of (\ref{eq:Res_wc}) maps tiltings to injectives. So does the left hand side.
Since  $\A^{0}_{\lambda^-,\lambda-\lambda^-}(n,r)\otimes_{\A_{\lambda^-}(n,r)}\bullet$ is an
equivalence $\OCat_\nu(\A_{\lambda}^-(n,r))\operatorname{-tilt}\xrightarrow{\sim} \OCat_\nu(\A_{\lambda}(n,r))\operatorname{-inj}$, we see that $\Res_\tau$ maps injectives to
injectives.

{\it Step 4}. To see that $\Res_\tau$ maps projectives to projectives, one can argue similarly
by using the functor $=\WC_{\lambda^-\leftarrow \lambda}^{-1}=R\Hom_{\A_{\lambda^-}(n,r)}(\A_{\lambda,\lambda^--\lambda}(n,r),\bullet)$
instead of $\A^{0}_{\lambda^-,\lambda-\lambda^-}(n,r)\otimes^L_{\A_{\lambda^-}(n,r)}\bullet$.
\end{proof}

%Now let \begin{align*}&\WC: D^b(\A_\lambda(n,r)\operatorname{-mod})\rightarrow %D^b(\A_{\lambda^-}(n,r)\operatorname{-mod}),\\
%&\WC_\tau:D^b(\A_\lambda(\tau,r)\operatorname{-mod})\rightarrow %D^b(\A_{\lambda^-}(\tau,r)\operatorname{-mod})\end{align*}
%be the wall-crossing functors.
%
%\begin{Lem}
%We have an isomorphism of functors $\WC_\tau\circ \Res_\tau\cong \Res_\tau\circ \WC$.
%\end{Lem}
%\begin{proof}
%Recall that $\WC$ is given by $$\A^{(-\theta)}_{\lambda,\lambda^--\lambda(n,r)}\otimes^L_{\A_\lambda(n,r)}\bullet,$$
%while $\WC_{\tau}$ is given by %$$\A^{(-\theta)}_{\lambda,\lambda^--\lambda}(n,r)_{\dagger,\tau}\otimes^L_{\A_\lambda(\tau,r)}\bullet,$$
%see \ref{?}. It is easy to see from the construction of $\Res_\tau$ that there is a bifunctorial
%isomorphism $\B_{\dagger,\tau}\otimes_{\A_\lambda(\tau,r)}\Res_\tau(M)\cong \Res_\tau(\B\otimes_{\A_\lambda(n,r)}M)$.
%This finishes the proof.
%\end{proof}

%Let us state some important corollaries of the lemmas above.
%
%\begin{Cor}\label{Cor:hw}
%The functor $\Res_\tau$ has the following properties:
%\begin{enumerate}
%\item It maps standard objects to standardly filtered ones.
%\item It maps costandard objects to costandardly filtered ones.
%\item It maps tilting objects to tilting objects.
%\item It maps projective objects to projective objects.
%\item It maps injective objects to injective objects.
%\end{enumerate}
%\end{Cor}

\subsection{Supports for asymptotic chambers}
Recall that the simples of $\OCat_{\nu,\lambda}(n,r)$ are indexed by the $r$-multipartitions $\sigma=(\sigma^{(1)},
\ldots,\sigma^{(r)})$ of $n$, moreover, $L_{\nu}(\sigma)=L_{\nu_0}(\underline{L}(\sigma))$, where
$\underline{L}(\sigma)=L^A(\sigma^{(1)})\boxtimes\ldots \boxtimes L^A(\sigma^{(r)})$.
We write $L^A(\sigma')$ for the irreducible module in the category $\mathcal{O}$ for the Rational
Cherednik algebra $eH_{1,\lambda}(|\sigma'|)e$, here $\sigma'$ is a partition.

Here is our main result of this section, it describes the dimensions of the supports of the simples
$L_{\nu}(\sigma)$ in terms of $\sigma$.

\begin{Thm}\label{Thm:count}
Let $m$ denote the denominator of $\lambda$ (equal to $+\infty$ if $\lambda\not\in \mathbb{Q}$).
Assume $m>1$.
Divide $\sigma^{(1)}$ by $m$ with remainder: $\sigma^{(1)}=m\sigma^{(1)q}+\sigma^{(1)r}$
(componentwise operations, ``q'' and ``r'' stand for the quotient and the remainder).
Then $\dim \Supp L_{\nu}(\sigma)= rn-|\sigma^{(1)q}|(rm-1)$.
\end{Thm}
\begin{proof}
Our proof is by induction on $n$. For $n<m$, the category $\mathcal{O}$ is semisimple, and all
simples have support of dimension $rn$, while for $n=m$, the result follows from
Theorem \ref{Thm:catO_str}. So we will assume that the claim of the theorem is proved
for all dimensions less than $n$.

Let us write $n_i$ for $|\sigma^{(i)}|$.
By \cite[Theorem 1.2]{B_ineq}, all irreducible components of $\Supp L_\nu(\sigma)$
have the same dimension. So if $\Res_\tau(L_\nu(\sigma))\neq 0$, then
$\dim \Supp L_\nu(\sigma)=\dim \Res_\tau(L_\nu(\sigma))$.
Recall, Corollary \ref{Cor:Res_simple}, that
$$\Res_\tau(L_\nu(\sigma))=L_\nu(\underline{\Res}_\tau\left(L^A(\sigma^{(1)})\boxtimes
L^A(\sigma^{(2)})\boxtimes\ldots\boxtimes L^A(\sigma^{(r)})\right)).$$

Let us first take $\tau=(n-1,1)$.
Let us compute $\underline{\Res}_\tau(L^A(\sigma^{(1)})\boxtimes
L^A(\sigma^{(2)})\boxtimes\ldots\boxtimes L^A(\sigma^{(r)}))$.
According to (\ref{eq:underline_R_comput}) it equals
$$\bigoplus_{i=1}^r L^{(A)}(\sigma^{(1)})\boxtimes\ldots\boxtimes E L^{(A)}(\sigma^{(i)})\boxtimes\ldots\boxtimes
L^{(A)}(\sigma^{(r)}),$$ where we write $E$ for the Bezrukavnikov-Etingof functor
restricting from $S_{n_i}$ to $S_{n_i-1}$. It is (up to a category equivalence)
the sum of categorification functors $E_i$ for a categorical action of $\hat{\mathfrak{sl}}_m$ on
$\bigoplus_{k=0}^{+\infty}\mathcal{O}(H_{\lambda+i-1}(k))$. Let us divide $\sigma^{(i)}$ by $m$ with remainder:
$\sigma^{(i)}=m\sigma^{(i)q}+\sigma^{(i)r}$. We have $EL^A(\sigma^{(i)})=0$
if and only if $\sigma^{(i)r}=0$, \cite{Wilcox}. Moreover, if $\sigma^{(i)r}\neq 0$, then
$EL(\sigma^{(i)})$ surjects onto $L(\underline{\sigma}^{(i)})$, where $\underline{\sigma}^{(i)}$
is obtained from $\sigma^{(i)}$ by removing a box from $\sigma^{(i)r}$, this follows from
results of   \cite{cryst}.

So assume that $\sigma^{(i)r}\neq \varnothing$ for some $i$. Let $\underline{\sigma}$ be the $r$-partition obtained from
$\sigma$ by replacing $\sigma^{(i)}$ with $\underline{\sigma}^{(i)}$. Then $\Res_\tau(L(\sigma))
\twoheadrightarrow L(\underline{\sigma})\boxtimes L$, where $L$ is some simple object in
$\OCat_\nu(\A_\lambda(1,r))$. By our inductive assumption, $\dim \Supp L(\underline{\sigma})\boxtimes L=
rn-(rm-1)|\sigma^{(1)q}|$. It follows that $\dim \Supp L(\sigma)=\dim \Res_\tau (L(\sigma))
\geqslant rn-(rm-1)|\sigma^{(1)q}|$.
On the other hand,  if $L^A(\hat{\sigma}^{(i)})$ is a simple occurring in $EL^A(\sigma^{(i)})$,
then $|\hat{\sigma}^{(i)q}|\geqslant |\sigma^{(i)q}|$. From here  we deduce that
$\dim \operatorname{Supp}\Res_\tau(L(\sigma))\leqslant rn-(rm-1)|\sigma^{(1)q}|$.

So we only need to consider the case when $\Res_\tau(L(\sigma))=0$. By the previous paragraph,
this means that all $\sigma^{(i)}$ are divisible by $m$. Take $\tau'=(m^{n/m})$.
Then, since the support of $L^A(\sigma^{(i)})$ corresponds to the partition $m^{n_i/m}$
of $n_i$, we can use (\ref{eq:underline_R_comput}) to get
$$\underline{\Res}_{\tau'}\left(L^A(\sigma^{(1)})\boxtimes\ldots \boxtimes L^A(\sigma^{(r)})\right)=\Res L^A(\sigma^{(1)})
\boxtimes\ldots\boxtimes \Res L^A(\sigma^{(r)}),$$
where we write $\Res$ for the Bezrukavnikov-Etingof restriction functor from $S_{mn_i}$ to $S_m^{n_i/m}$.
According to \cite{Wilcox}, we have $\Res L^A(\sigma^{(i)})=L^A(m)^{\boxtimes n_i/m}\boxtimes V_{\sigma^{(i)q}}$,
where $V_{\sigma^{(i)q}}$ is a nonzero multiplicity space (that is the irreducible $S_{|\sigma^{(i)q}|}$-module
corresponding to $\sigma^{(i)q}$).

It follows that $\Res_{\tau'}L(\sigma)$ is the sum of several  copies of
$L_1^{\boxtimes n_1/m}\boxtimes L_2^{\boxtimes n_2/m}\boxtimes\ldots\boxtimes L_r^{\boxtimes n_r/m}$.
Here we write $L_i$ for the irreducible in $\OCat_\nu(\A_\lambda(m,r))$ corresponding to
the $r$-partition with $i$th part $(m)$. We have $\dim \Supp L_1=1$ and $\dim \Supp L_i=rm$.
Since $\dim \Supp L(\sigma)=\dim \Supp\Res_{\tau'}L(\sigma)$, this implies the claim of
the theorem in this case and finishes the proof.
\end{proof}

\begin{proof}[Proof of Theorem \ref{Thm:support}]
In the proof we can assume that $\lambda$ is sufficiently big. By \cite[Theorem 1.2]{B_ineq},
we have $\dim \Supp L_\nu(\sigma)=\frac{1}{2}\dim \VA(\A_\lambda(n,r)/\operatorname{Ann}L_\nu(\sigma))$.
Theorem \ref{Thm:ideals} implies that the two-sided ideals in $\A_\lambda(n,r)$ are determined
by the dimensions of their associated varieties. This finishes the proof of this theorem.
\end{proof}

\subsection{Cross-walling bijections}
We have computed the supports in the case when $\nu$ is dominant. For other chambers of
one-parameter subgroups, supports can be computed using the cross-walling bijections
defined in a more general situation below. We plan to compute these bijections combinatorially
in a subsequent paper.

Let $X=X^\theta$ be a conical symplectic resolution equipped with a Hamiltonian action
of a torus $T$ with  finitely many fixed points.  Let $\mathcal{\A}_\lambda^\theta$ be a quantization of $X$
and $\A_\lambda:=\Gamma(\A_\lambda^\theta)$. We suppose that  $\lambda$ is sufficiently ample
so that Propositions \ref{Prop:cat_O} and \ref{Prop:stand_filt} hold.

Let $\nu,\nu':\C^\times\rightarrow T$ be two generic one-parameter subgroups lying in
chambers opposite with respect to a face. We are going to define a bijection $\mathfrak{cw}_{\nu'\leftarrow\nu}:X^T\rightarrow X^T$. Let $\nu_0$ be a generic one-parameter subgroup in the common face of the chambers of $\nu,\nu'$ such that these chambers are opposite with respect to the face.
Consider the cross-walling functor
$\underline{\CW}_{\nu'\leftarrow \nu}:D^b(\OCat_{\nu}(\Ca_{\nu_0}(\A^\theta_\lambda)))\rightarrow
D^b(\OCat_{\nu'}(\Ca_{\nu_0}(\A^\theta_\lambda)))$. For $X^{\nu_0(\C^\times)}$, the one-parameter subgroups
$\nu,\nu'$ lie in  opposite chambers. So $\underline{\CW}^{-1}_{\nu'\leftarrow \nu}$
is the direct sum (over the irreducible components of $X^{\nu_0(\C^\times)}$) of Ringel duality functors
with various homological shifts. Each Ringel duality functor is a perverse equivalence, Lemma \ref{Lem:WC_Ringel},
and hence gives rise to a bijection between the set of simples, see Section \ref{SS_long_WC}. We take the disjoint
union of these bijections for $\mathfrak{cw}_{\nu'\leftarrow \nu}$.

The  main result of this section is the following.

\begin{Prop}\label{Prop:cw_dim_preserv}
We have $\operatorname{Ann}(L_\nu(p))=
\operatorname{Ann}(L_{\nu'}(\mathfrak{cw}_{\nu'\leftarrow \nu}(p)))$
and $\dim \Supp L_\nu(p)=\dim \Supp L_{\nu'}(\mathfrak{cw}_{\nu'\leftarrow \nu}(p))$.
\end{Prop}
\begin{proof}
Note that the former equality implies the latter by Lemma \ref{Lem:O_holon}.

Let us prove the equality of the annihilators. Recall that $\CW_{\nu'\leftarrow \nu}\circ\Delta_{\nu_0}
\cong \Delta_{\nu_0}\circ \underline{\CW}_{\nu'\leftarrow \nu}$, Proposition \ref{Prop:cw_functors}.  By the construction, $\underline{L}_{\nu'}(\mathfrak{cw}_{\nu'\leftarrow \nu}(p))$ is a constituent in
$H_{\bullet}(\underline{\CW}_{\nu'\leftarrow \nu}\underline{L}_\nu(p))$.
It follows that  $L_{\nu'}(\mathfrak{cw}_{\nu'\leftarrow \nu}(p))$ is a constituent in
$H_{\bullet}(\CW_{\nu'\leftarrow \nu}L_\nu(p))$. Set $\I:=\bigcap_n \operatorname{Ann}L_\nu(p)^n$
so that $\I^2=\I$.
Since the regular $\A_\lambda$-bimodule has finite length, see Lemma \ref{Lem:fin_length},
we see that $\I$ coincides with some power of $\operatorname{Ann}L_\nu(p)$.
It is enough to show that
$H_\bullet(\CW_{\nu'\leftarrow \nu}L_\nu(p))$ is annihilated by $\I$.

For any HC $\A_\lambda$-bimodule $\B$ we have $\B\otimes^L_{\A_\lambda} \CW_{\nu'\leftarrow \nu}(\bullet)
\cong \CW_{\nu'\leftarrow \nu}(\B\otimes^L_{\A_\lambda}\bullet)$. This is because the functors
$\mathcal{B}\otimes^L_{\A_\lambda}\bullet, R\operatorname{Hom}_{\A_\lambda}(\mathcal{B},\bullet)$
preserve the categories $D^b(\mathcal{O}_{\nu''})$ for all generic $\nu''$.

We will apply the previous paragraph to $\B=\I$. We have
\begin{equation}\label{eq:cw_ideal_prod}\mathcal{I}^{\otimes^L k}\otimes^L \CW_{\nu'\leftarrow \nu}(L_{\nu}(p))\cong
\CW_{\nu'\leftarrow \nu}(\mathcal{I}^{\otimes^L,k}\otimes^L L_{\nu}(p)),\end{equation}
where all derived tensor products are taken over $\A_{\lambda}$ and we
write $\mathcal{I}^{\otimes^L k}$ for the $k$-th derived tensor power of
$\I$. Note that thanks to $\I^2=\I$ we get $\I\otimes_{\A_\lambda}\A_\lambda/\I=0$. It follows that
$\I\otimes_{\A_\lambda}M=0$ for all $\A_\lambda$-modules $M$ annihilated by $\I$. Moreover,
thanks to the exact sequence $0\rightarrow \mathcal{I}\rightarrow \A_\lambda
\rightarrow \A_\lambda/\I\rightarrow 0$, all homology of $\I\otimes^L_{\A_\lambda}M$
are annihilated by $\I$ provided $M$ is annihilated by $\I$. We deduce that
$H_i(\mathcal{I}^{\otimes^L,k}\otimes^L L_{\nu}(p))=0$ for $i< k$.
We deduce that the homology of the right hand side of (\ref{eq:cw_ideal_prod})
vanishes in degrees $\leqslant k-1$. On the other hand, let
$k$ be the minimal number such that $H_\ell(\CW_{\nu'\leftarrow \nu}(L_{\nu}(p)))$
is not annihilated by any power of $\I$. Using the spectral sequence for the composition
of derived functors, we see that $H_\ell$ of the left hand side of (\ref{eq:cw_ideal_prod})
is nonzero for any $k$. This gives a contradiction that completes the proof.
\end{proof}

\end{document}